\documentclass[12pt]{amsart}

\usepackage{amssymb}

\usepackage{enumitem}

\usepackage{graphicx}
\usepackage{color}

\usepackage{mathrsfs}

\usepackage{comment}

\usepackage{url}

\newcommand{\R}{{\mathbb R}}

\newcommand{\N}{{\mathbb N}}

\newcommand{\diam}{{\textnormal{diam}}}
\newcommand{\cl}{{\textnormal{cl}}}
\newcommand{\interior}{{\textnormal{int}}}

\makeatletter
\@namedef{subjclassname@2020}{%
  \textup{2020} Mathematics Subject Classification}

\usepackage[T1]{fontenc}

\newtheorem{theorem}{Theorem}[section]
\newtheorem{corollary}[theorem]{Corollary}
\newtheorem{lemma}[theorem]{Lemma}
\newtheorem{proposition}[theorem]{Proposition}

\theoremstyle{definition}
\newtheorem{definition}[theorem]{Definition}
\newtheorem{remark}[theorem]{Remark}

\numberwithin{equation}{section}

\begin{document}


\baselineskip=17pt


\title[Continuum Many Planar Embeddings of the Knaster accordion]{Continuum Many Planar Embeddings of the Knaster accordion}

\author[Joseph S. Ozbolt]{Joseph S. Ozbolt}
\thanks{This paper comprises the results of the author's dissertation, written under the direction of Piotr Minc and presented to the faculty of Auburn University in partial fulfillment of the requirements for the degree of Doctor of Philosophy, August 2020.}
\address{Joseph S. Ozbolt \\ COLSA Corporation\\ 
6728 Odyssey Dr NW\\
Huntsville, Alabama 35806}
\email{joesozbolt@gmail.com}

\date{}

\begin{abstract}
A\v nusi\'c, Bruin, and \v Cin\v c have asked in \cite{Anusic2} which hereditarily decomposable chainable continua (HDCC) have uncountably many mutually inequivalent planar embeddings. It was noted, as per the embedding technique of John C. Mayer with the $\sin(1/x)$-curve \cite{Mayer}, that any HDCC which is the compactification of a ray with an arc likely has this property. Here, we show two methods for constructing $\mathfrak{c}$-many mutually inequivalent planar embeddings of the classic Knaster $V \Lambda$-continuum, $K$, also referred to as the Knaster accordion. The first of these two methods produces $\mathfrak{c}$-many planar embeddings of $K$, all of whose images have a different set of accessible points from the image of the standard embedding of $K$, while the second method produces $\mathfrak{c}$-many embeddings of $K$, each of which preserve the accessibility of all points in the standard embedding.
\end{abstract}



\maketitle

\section{Introduction and Preliminaries}  
When referring to a collection as having \textit{continuum many} elements, we mean that its cardinality is the same as that of the set of real numbers, which will be denoted by $\mathfrak{c}$. By \textit{a continuum}, we mean a compact connected metric space. 

Let $X$ be a metric space with metric $d$, and let $A \subset X$. The \textit{diameter} of $A$, denoted $\diam(A)$, is given by $\diam(A) = \sup\{d(x,y) \mid x, y \in A\}.$ A finite collection $\mathcal{C} = \{C_1, C_2, \ldots, C_n\}$ of subsets of $X$ having the property that $C_i \cap C_j \neq \emptyset$ if and only if $|i-j| \leq 1$ is called a \textit{chain} in $X$. By the \textit{mesh} of $\mathcal{C}$, we mean the maximum of the set of diameters of each member $\mathcal{C}$. If $\mathcal{C}$ covers $X$, that is, if $X = \bigcup \mathcal{C}$, we say that $\mathcal{C}$ is a \textit{chain covering} of $X$.

A continuum is \textit{chainable} if it can be covered by a chain covering of open subsets having arbitrarily small mesh. If a continuum can be expressed as the union of two of its proper subcontinua, then it is said to be \textit{decomposable}; otherwise, we say it is \textit{indecomposable}.  If a continuum has the property that each of its nondegenerate subcontinua is decomposable, it is said to be \textit{hereditarily decomposable}. We will mainly concern ourselves with hereditarily decomposable chainable continua (HDCC for both singular and plural).

By \textit{the plane}, we mean here the $xy$-plane, which will sometimes be denoted by $\R^2$, endowed with the usual Euclidean metric, which will be denoted as $d$. Given a planar continuum $X$ and two planar embeddings $\varphi$ and $\psi$ of $X$, we say that $\varphi$ and $\psi$ are \textit{equivalent planar embeddings of $X$} if there exists a homeomorphism of the plane onto itself mapping $\varphi(X)$ onto $\psi(X)$. If no such homeomorphism exists, we say that $\varphi$ and $\psi$ are \textit{inequivalent planar embeddings of $X$}, or simply, are inequivalent for short. If $\Phi$ is a collection of embeddings of $X$ having the property that for each $\varphi$ and $\psi$ in $\Phi$ where $\varphi \neq \psi$, $\varphi$ and $\psi$ are inequivalent, then we say that $\Phi$ is a \textit{collection of mutually inequivalent embeddings of $X$}.

Let $X$ be a continuum in the plane and let $x \in X$. If there exists an arc $A$ in the plane such that $A \cap X = \{x\}$, we say $x$ is \textit{accessible from the complement of $X$}. For short, we may say that $x$ is an accessible point of $X$, or simply, that $x$ is accessible if the continuum $X$ is understood.

If $x$ is an accessible point of $\varphi(X)$ and $A$ an arc contained in the plane having $x$ as an endpoint with $A \cap \varphi(X) = \{x\}$, then $A$ is called an \textit{endcut} of $\varphi(X)$. If $C$ is an arc contained in the plane whose interior is contained in the complement of $\varphi(X)$ and whose endpoints are contained in $\varphi(X)$, then $C$ is called a \textit{crosscut} of $\varphi(X)$.

Planar embeddings of continua have been a subject of inquiry in topology since at least the early twentieth century.  According to R.H. Bing in \cite{Bing}, Theorem 4, every chainable continuum can be embedded in the plane. It is well-known, as one can verify is a consequence of the Jordan-Schoenflies Theorem \cite{Cairns}, that an arc has only one planar embedding up to equivalence. The question of how many mutually inequivalent planar embeddings can be produced for one specific chainable continuum or a particular class of chainable continua has also been of interest. For example, Michel Smith \cite{Smith} and Wayne Lewis \cite{Lewis} have both independently shown that there are uncountably many mutually inequivalent planar embeddings of the pseudo-arc. More recent results include those of Anu\v si\' c, \v Cin\v c, and Bruin in \cite{Anusic1}, \cite{Anusic2}, and \cite{Anusic4}.

The accessibility of points in the images of planar embeddings of continua is of particular interest, of which a well-known problem is that of Nadler and Quinn in \cite{Nadler-Quinn}. It is asked, given a chainable continuum $X$ and a point $x \in X$, if there is a planar embedding $\varphi$ of $X$ such that $\varphi(x)$ is accessible. This question was answered as positive for all HDCC by Minc and Transue in Theorem 6.1 of \cite{Minc Transue}. However, the question remains open for indecomposable continua. A survey of this problem can be found in \cite{Anusic3}. How accessibility of points in the images of planar embeddings relates to inequivalent embeddings can be manifested by the following proposition, which says that if the images of two planar embeddings of a continuum have a different set of accessible points, then the embeddings are inequivalent.

\begin{proposition}\label{different access}
Let $\varphi$ and $\psi$ be planar embeddings of a continuum $X$. If there is an $x \in X$ such that $\varphi(x)$ is accessible while $\psi(x)$ is not, then $\varphi$ and $\psi$ are inequivalent.
\end{proposition}

Thus, given a planar continuum $X$, one can provide a pair of mutually inequivalent planar embeddings $\varphi$ and $\psi$ of $X$ by constructing $\varphi$ and $\psi$ so that at least one point of $X$ is accessible in the image $\varphi$ and inaccessible in the image of $\psi$. In \cite{Mayer}, John C. Mayer constructed a procedure for embedding the $\sin(1/x)$-continuum (shown in Figure \ref{fig:sine}) in uncountably many mutually inequivalent ways so that their images all have the same set of accessible points. This procedure is executed by first developing a \textit{schema} (plural: \textit{schemata}) consisting of \textit{subschemata} (singular: \textit{subschema}) for each sequence of nonnegative integers, followed by manipulating the ray in the $\sin(1/x)$-continuum on each side of its limiting arc according to the rules of the schema.  A somewhat similar procedure can apply to forming uncountably many mutually inequivalent embeddings of other HDCC.  One example in which we demonstrate using such a similar embedding procedure will be with the Knaster $V\Lambda$-Continuum, otherwise known as the \textit{Knaster accordion}, which we will denote as $K$, shown in Figure \ref{fig:Knaster}. This paper will be devoted to $K$, in which we construct two different collections of mutually inequivalent embeddings of $K$ each having cardinality $\mathfrak{c}$. This exhibits a way to produce uncountably many, and in fact $\mathfrak{c}$-many, mutually inequivalent planar embeddings of an HDCC which contains no subcontinuum having a dense ray and which is not path connected and is nowhere locally connected.

\begin{figure}
\centering
\includegraphics[width=12cm]{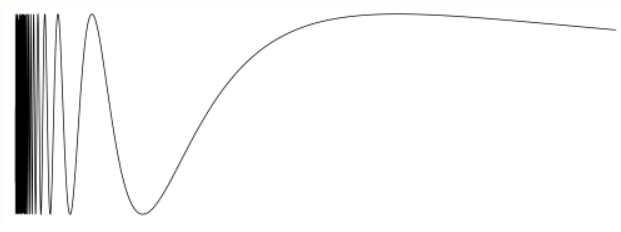}
\caption{The $\sin(1/x)$-continuum, also known as the topologist's sine curve, is an HDCC which is the compactification of a ray with an arc. This continuum is not path connected, nor is it locally connected.}
\label{fig:sine}
\end{figure}

Recently, A\v nusi\'c, Bruin, and \v Cin\v c showed in \cite{Anusic2} that every chainable continuum containing a nondegenerate indecomposable subcontinuum admits uncountably many mutually \textit{strongly} inequivalent planar embeddings.\footnote{Planar embeddings $\varphi$ and $\psi$ of a planar continuum $X$ are strongly equivalent if $\psi \circ \varphi^{-1}: \varphi(X) \to \psi(X)$ can be extended to a homeomorphism of the plane onto itself.} They asked (Question 6 of \cite{Anusic2}) which HDCC (other than an arc) have uncountably many mutually inequivalent planar embeddings. It was noted that Mayer's embedding approach for the $\sin(1/x)$-continuum, as mentioned in the previous paragraph, likely works to show that any continuum which is the compactification of a ray with an arc has uncountably many mutually inequivalent planar embeddings. They further added that Mayer's approach would not generalize to the remaining HDCC since not all have subcontinua with dense rays.

As previously noted, and as made apparent in Figure \ref{fig:Knaster}, $K$ is an HDCC containing no subcontinuum which is the compactification of a ray with an arc. In fact, $K$ contains no subcontinuum containing a dense ray. Furthermore, $K$ is not path connected and is also nowhere locally connected. Note that the $\sin(1/x)$-continuum contains a dense ray, and although it is not path connected, it is not nowhere locally connected. Regardless, we will still be able to give a "Mayer-like" approach to producing uncountably many, and in fact, $\mathfrak{c}$-many, mutually inequivalent planar embeddings of $K$.  Furthermore, we will provide two methods of producing these embeddings---one which produces $\mathfrak{c}$-many mutually inequivalent planar embeddings of $K$, all of whose images have a different set of accessible points from the image of the standard embedding of $K$, and the other producing $\mathfrak{c}$-many mutually inequivalent planar embeddings of $K$ which preserve the accessibility of points accessible in the standard embedding. The collection of embeddings of the former will be constructed in Section 2 and the latter will be constructed in Section 3.

Before we give a geometric construction defining $K$, we must first provide some basic theory and terminology regarding the general structure of all HDCC, most of which is extracted from Kuratowski's theory on the structure of irreducible continua in Chapter V, \S48 of \cite{Kuratowski}. For a given HDCC $X$, there exists a continuous function $g$ mapping $X$ onto $[0,1]$ so that for each $t \in [0,1]$, $g^{-1}(t)$ is a maximal nowhere dense subcontinuum of $X$. Such a function $g$ will be called a \textit{Kuratowski map} of $X$, and the subcontinua $g^{-1}(t)$ will be called the \textit{layers} of $X$. The layers given by $g^{-1}(0)$ and $g^{-1}(1)$ will be called the \textit{left end layer} and \textit{right end layer} of $X$, respectively.  All other layers are called \textit{interior layers} of $X$.

In particular, each HDCC admits an upper semicontinuous decomposition into layers. That is, if $X$ is an HDCC and $g$ a Kuratowski map of $X$, then $g^{-1}$ is an upper semicontinuous set-valued map. The curious reader may refer to section III, \S 2 of \cite{Nadler} for more information on upper semicontinuous decompositions of continua.  Each nondegenerate layer $L$ of an HDCC $X$ is also an HDCC and can itself be decomposed into layers. As can be found in \cite{Minc Transue}, \cite{Mohler}, and \cite{Thomas}, we refer to \textit{generalized layers} of $X$ as those which may be layers of $X$, layers of layers of $X$, etc. More specifically, let $\mathcal{L}_0(X) = \{X\}$. If $\alpha = \beta + 1$, let $\mathcal{L}_\alpha(X)$ consist of the degenerate members as well as the layers of the nondegenerate members of $\mathcal{L}_\beta(X)$. If $\alpha$ is a limit ordinal, let $\mathcal{L}_\alpha(X) = \{\bigcap_{\beta < \alpha}L_\beta \mid L_\beta \in \mathcal{L}_\beta(X)\}$. It was also shown by Thomas in \cite{Thomas} that there exists a least countable ordinal $\sigma_X$ such that every member of $\mathcal{L}_{\sigma_X}(X)$ is degenerate. Thus, the generalized layers of $X$ are any members of the set $\mathscr{L}(X) = \bigcup_{\alpha \leq \sigma_X} \mathcal{L}_\alpha(X)$. Mohler showed in \cite{Mohler} that for every countable ordinal $\alpha$, there exists an HDCC $X_\alpha$ so that every member of $\mathcal{L}_\alpha(X_\alpha)$ is degenerate. Thus, for a given HDCC $X$ and a point $x \in X$, there exists a countable ordinal $\sigma_x$ such that $\sigma_x = \min\{\alpha \mid \{x\} \in \mathcal{L}_\alpha(X)\}$. Note that if $L$ is a nondegenerate generalized layer of $X$, then there exists a unique countable ordinal $\sigma_L > 0$ such that $L \in \mathcal{L}_{\sigma_L}(X)$.  In either case, we say that $\sigma_X$ is the \textit{layer level of $X$}, that $\sigma_x$ is the \textit{layer level of $x$ in $X$} for $x \in X$, and that for a nondegenerate generalized layer $L$ of $X$, $\sigma_L$ is the \textit{layer level of $L$ in $X$}. It becomes self-evident at this point that when referring to a layer of an HDCC, we mean a generalized layer having a layer level of 1. We shall also refer to such layers as \textit{top layers} to avoid any possible ambiguity in a given context.
It is worth noting that generalized layers of $K$ having a layer level more than 1 will not be used in this paper. However, we have defined generalized layers of HDCC since they are relevant in one of the questions of section 4.

A nondegenerate continuum $Y$ is \textit{irreducible between two points} $x, y, \in Y$ if there is no proper subcontinuum of $Y$ containing both $x$ and $y$.  Likewise, $Y$ is \textit{irreducible between two subcontinua} $A_1, A_2 \subset Y$ if there is no proper subcontinuum $Z$ such that $A_1, A_2 \subset Z$. It is well-known (again, see Chapter V, \S 48 of \cite{Kuratowski}) that all HDCC are irreducible between two points. In fact, if $Y$ is an HDCC and $x,y \in Y$ are any two distinct points (subcontinua) of $Y$, then there is a subcontinuum of $Y$ irreducible between $x$ and $y$. 

If $X$ is an HDCC and $x$ and $y$ are distinct points in $X$, then the subcontinuum of $X$ irreducible between $x$ and $y$ will be denoted by $[x,y]$. Likewise, given subcontinua $C_1$ and $C_2$ of $X$, we denote the subcontinuum of $X$ irreducible between $C_1$ and $C_2$ as $[C_1, C_2]$. Also, we have $(C_1, C_2] := [C_1, C_2]\backslash C_1$, $[C_1, C_2) := [C_1, C_2]\backslash C_2$, and $(C_1, C_2) := [C_1, C_2] \backslash (C_1 \cup C_2)$. Note that if $C_1$ and $C_2$ are top layers of $X$ with $x \in C_1$ and $y \in C_2$, then $[x,y] = [C_1, C_2]$, $(x,y] = (C_1, C_2]$, $[x,y) = [C_1, C_2)$, and $(x,y) = (C_1, C_2)$. 

One may note that the interval-like notation provided in the previous paragraph for such irreducible subcontinua implies there is some assigned ordering. In some cases, one can naturally infer the ordering based on a Kuratowski map $g: X \to [0,1]$, in that if $x$ and 
$y$ are distinct points of $X$ such that $g(x) < g(y)$, then the subcontinuum of $X$ irreducible between $x$ and $y$ is $[x,y]$. Likewise, if $C_1$ and $C_2$ are distinct subcontinua of $X$ such that
$$\min\{g(x) \mid x \in C_1\} < \min\{g(x) \mid x \in C_2\},$$
then the subcontinuum of $X$ irreducible between $C_1$ and $C_2$ is $[C_1,C_2]$. For the purposes of this paper, this will always be the case. However, in cases where points $x$ and $y$ or subcontinua $C_1$ and $C_2$ are contained within the same generalized layer, one must more carefully define a consistent ordering on $X$. The reader is referred to \S 3 of \cite{Minc and Transue 2} for a detailed description of such an ordering.

Given an interior layer $L$ of $X$, we define the \textit{left part of $L$} and the \textit{right part of $L$} as 
$$\ell(L) = \cl\Big(g^{-1}\big([0, g(L))\big)\Big) \cap L \ \textnormal{ and } \ r(L) = L \cap \cl\Big(g^{-1}\big((g(L), 1]\big)\Big),$$
respectively.

\begin{definition}\label{layers of cohesion}
Let $X$ be an HDCC, let $g$ be a Kuratowski map of $X$, and let $L$ be a layer of $X$. We say that $L$ is a \textbf{layer of cohesion} if $L$ is an end layer or if
$\ell(L) = L = r(L)$.
\end{definition}

\begin{definition}\label{layers of continuity}
Let $X$ be an HDCC and let $g$ be a Kuratowski map of $X$. We say that a layer $L$ of $X$ is a \textbf{layer of continuity} if the set-valued map $g^{-1}$ is continuous at the point $g(L)$.
\end{definition}

As noted by Kuratowski on page 201 of \cite{Kuratowski}, layers of cohesion need not be layers of continuity, as is the case for the limiting arc of the $\sin(1/x)$-continuum.

\begin{proposition}\label{preserve order}
Let $X$ be a hereditarily decomposable chainable continuum, let $h$ be a homeomorphism of $X$ onto itself, and let $g: X \to [0,1]$ be a Kuratowski map.  Then for every $s, t \in [0,1]$ such that $s \leq t$, we have either $g(h(g^{-1}(s))) \leq g(h(g^{-1}(t)))$ or $g(h(g^{-1}(s))) \geq g(h(g^{-1}(t)))$.
\end{proposition}

That is, Proposition \ref{preserve order} says that any homeomorphism of an HDCC onto itself will preserve the order of the top layers in the sense that if one top layer is between two other top layers, the same will hold in its image under any homeomorphism.

\begin{figure}
\centering
\includegraphics[width=10cm]{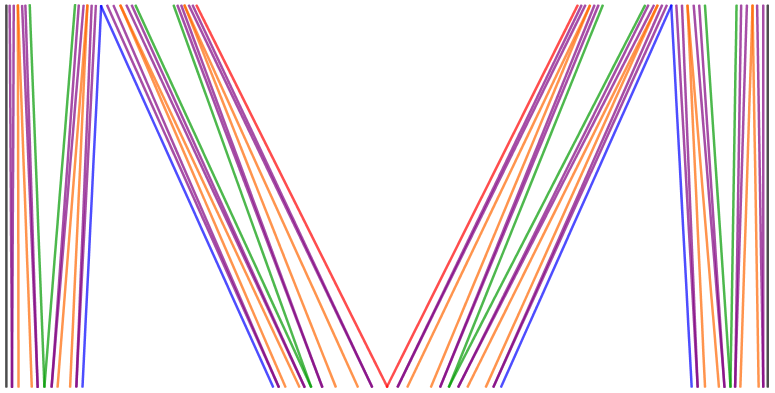}
\caption{The (horizontally elongated) image of the standard embedding of the Knaster $V \Lambda$-continuum, $K$, also known as the Knaster accordion, approximated by the first five steps along with the end layers. Note that $K$ contains no subcontinuum containing a dense ray and that it is also not path connected and is nowhere locally connected.}
\label{fig:Knaster}
\end{figure}

We may now begin to construct the Knaster $V \Lambda$-Continuum, otherwise called the Kuratowski accordion, which will be denoted by $K$ and shown in Figure \ref{fig:Knaster}. By a "$V$" and a "$\Lambda$," we mean an arc consisting only of two maximally straight line segments shaped like the letters $V$ and $\Lambda$, respectively.\footnote{More precisely, we mean arcs shaped like $\lor$ and $\land$, respectively.} We construct $K$ in such a way that also constructs what is referred here as \textit{the standard planar embedding of $K$}, or \textit{standard embedding of $K$} for short, whose (horizontally elongated) image is shown in Figure \ref{fig:Knaster}, as given by the following sequence of steps.
    \begin{itemize}
        \item[$(1.)$] Draw the $V$ whose vertex is the point $(1/2, 0)$ and whose endpoints are $(1/3, 1)$ and $(2/3, 1)$.
        \item[\vdots] 
        \item[$(n.)$] Consider the set $\Delta_n$ consisting of all $2^{n-1}$ quadrilaterals contained in $[0,1] \times [0,1]$ whose left and right sides contain either $\{0\} \times [0,1]$, $\{1\} \times [0,1]$, or maximal straight segments of two different $V$'s or $\Lambda$'s from the preceding steps, and whose interiors contain no points from any $V$'s or $\Lambda$'s in the preceding steps. If $n$ is even (odd), we draw the $2^{n-1}$-many $\Lambda$'s ($V$'s) each sitting in individual members of $\Delta_n$. The vertex of each $\Lambda$ ($V$) is on the top (bottom) side of its given quadrilateral, sitting halfway between the two top (bottom) vertices of the quadrilateral. Finally, the two endpoints of each $\Lambda$ ($V$) sit evenly spaced on the bottom (top) side of its given quadrilateral. 
        \item[\vdots]
        \item[$(\omega.)$] Let $K$ be the closure of the union of all $V$'s and $\Lambda$'s from each of the preceding steps.
    \end{itemize}
    
Again, a rough image of this standard embedding of $K$ can be seen in Figure \ref{fig:Knaster} as approximated by the first five steps listed above, together with the arcs $\{0\} \times [0,1]$ and $\{1\} \times [0,1]$. Generally, when referring to $K$ without specifying any particular embedding, it can be assumed we mean the image of the standard embedding of $K$. One may observe geometrically that $K$ is chainable. Upon taking the closure of the union of all $V$'s and $\Lambda$'s to form $K$ as described in step $\omega$, we have immediately inserted $\mathfrak{c}$-many straight arcs, two of which are the left end $\{0\} \times [0,1]$ and the right end $\{1\} \times [0,1]$, the rest of which lie between $V$'s and $\Lambda$'s. It can also be observed that these $\mathfrak{c}$-many straight arcs are the layers of continuity as well as the layers of cohesion of $K$, with $\{0\} \times [0,1]$ the left end layer and $\{1\} \times [0,1]$ the right end layer.  Each $V$ and $\Lambda$ of $K$ is layer of $K$ which is neither a layer of continuity nor a layer of cohesion. Thus, $K$ is an HDCC with each of its layers being arcs and thus nondegenerate. Furthermore, the layer level of $K$ is 2. Finally, one may observe geometrically from Figure \ref{fig:Knaster} that the set of accessible points of the standard embedding of $K$ consist of all points of end layers and non-cohesion layers of $K$ and the endpoints of interior layers of continuity of $K$. That is, if $L$ is either an end layer, a $V$-layer, or a $\Lambda$-layer, then every point of $L$ is accessible, whereas if $L$ is an interior straight-arc layer, then only the endpoints of $L$ are accessible.

Kazimierz Kuratowski attributed $K$ to Bronis\l aw Knaster in \cite{Kuratowski}, hence its namesake herein. It has also be referred to as the "Cajun accordion" by James Rogers in \cite{Rogers1} and \cite{Rogers2} and by David Lipham in \cite{Lipham}. In \cite{Rogers2}, James Rogers gives mention to both the Knaster accordion and its circularly chainable counterpart which he refers to as the "Zydeco accordion." The Zydeco accordion can be constructed by identifying the end layers of $K$. In particular, he points out that although many continuum theorists would consider the Zydeco accordion as being "rich" in rotations, it happens to be meager with respect to extendable intrinsic rotations about the origin.

Although vertices of layers of $K$ can be understood to be the geometric vertices of $V$ or $\Lambda$ layers of $K$, we also provide the following topological definition of vertices of layers of $K$.
\begin{definition}\label{vertex}
Let $L$ be a $V$-layer or a $\Lambda$-layer of $K$.  We say  that a point $v \in L$ is a \textbf{vertex} of $L$ if $v \in \interior_L(L)$ and for every open subset $U$ of $K$ containing $v$, $U$ contains infinitely many endpoints of layers of $K$. 
\end{definition}

If $v$ is a vertex of a layer $L$, it is understood that $L$ is not a layer of continuity of $K$, i.e., $L$ is either a $V$ or a $\Lambda$. Furthermore, we may also say that $v$ is a vertex of $K$, understanding that it is a vertex of a layer of $K$ which is not a layer of continuity of $K$.  Note that in the following lemma, we will let $E = \{0\} \times [0,1]$ and $E' = \{1\} \times [0,1]$ denote the end layers of $K$, we will let $p = (0,0)$ and $p' = (1,0)$ denote the bottom endpoints of $E$ and $E'$, respectively, and we will let $q = (0,1)$ and $q' = (1,1)$ denote the top endpoints of $E$ and $E'$, respectively.  Furthermore, we will denote the collection of bottom endpoints and vertices of layers of $K$ by $P$ while the top such points will be denoted by $Q$. That is, $P$ consists of all points in $K$ whose $y$-coordinates are 0 under the standard embedding of $K$ while $Q$ consists of all points in $K$ whose $y$-coordinates are 1 under the standard embedding of $K$.

\begin{proposition}\label{vertex to vertex}
Let $h$ be a homeomorphism of $K$ onto itself.  Then endpoints of layers of $K$ are mapped by $h$ to endpoints of layers of $K$ and vertices of $K$ are mapped by $h$ to vertices of $K$.
\end{proposition}
\begin{proof}
Since every layer of $K$ is an arc, the endpoints of layers of $K$ must be mapped to endpoints of layers of $K$.
If $v$ is a vertex of $K$, then it is in the interior of some non-continuity layer of $K$.  Thus, $h(v)$ must be contained the interior of some layer of $K$.  Furthermore, if $U$ is an open subset of $K$ containing $h(v)$, then $h^{-1}(U)$ is an open subset of $K$ containing $v$ and thus must contain an infinite number of endpoints of layers of $K$, whence $U$ must contain an infinite number of endpoints of layers of $K$. Therefore, $h(v)$ is also a vertex of $K$.
\end{proof}

\begin{corollary}\label{like layers}
If $h$ is a homeomorphism of $K$ onto itself, then every $\Lambda$-layer or $V$-layer of $K$ is mapped to a $\Lambda$-layer or $V$-layer of $K$ and every layer of continuity of $K$ is mapped to a layer of continuity of $K$.
\end{corollary}

\begin{lemma}\label{homeos of K}
Let $h$ be a homeomorphism of $K$ onto itself.  Then $h(\{E,E'\}) = \{E, E'\}$.  Furthermore, if $h(p) \in \{p, p'\}$, then $h(P) = P$ and $h(Q) = Q$, and if $h(p) \in \{q, q'\}$, then $h(P) = Q$ and $h(Q) = P$.
\end{lemma}
\begin{proof}
Let $g: K \to [0,1]$ be a Kuratowski map.  Then $g \circ h$ is also a Kuratowski map.  Thus, $h(\{E, E'\}) = h((g\circ h)^{-1}(\{0,1\})) = h(h^{-1}(g^{-1}(\{0,1\})) = g^{-1}(\{0,1\}) = \{E, E'\}$.


Suppose now that $h(p) \in \{p, p'\}$, but that there exists a nonempty subset $C$ of $P$ such that $h(c) \notin P$ for every $c \in C$. By Proposition \ref{vertex to vertex}, it follows that $h(c) \in Q$ for every $c \in C$.  Since $h$ is a homeomorphism, it follows that for every $c \in C$, there exists an open subset $U_c$ of $K$ containing $c$ such that $h(P \cap U_c) \subset Q$. Let $D = P\backslash C$---the set of all members of $P$ mapped into $P$ by $h$. Again, since $h$ is a homeomorphism, it follows that for every $d \in D$, there is an open subset $U_d$ of $K$ containing $d$ such that $h(P \cap U_d) \subset P$.
Note that $C = \bigcup_{c \in C}(P \cap U_c)$ is open in $P$ and that $D = \bigcup_{d\in D}(P \cap U_d)$ is also open in $P$.  Also, since $C = P\backslash D$, it follows that both $C$ and $D$ are also closed in $P$.

Let $g$ be a Kuratowski map of $K$ onto $[0,1]$ as before.  Since $g$ is continuous, $g(C)$ and $g(D)$ are both closed in $[0,1]$.  Since $[0,1] = g(C) \cup g(D)$ and $g(C) = [0,1]\backslash g(D)$ with $g(C)$ and $g(D)$ both nonempty, it follows that $g(C)$ and $g(D)$ are also open in $[0,1]$, making them both nonempty, open and closed subsets of $[0,1]$ whose union is $[0,1]$.  This contradicts that $[0,1]$ is connected.  Therefore, $h(P) = P$ and $h(Q) = Q$.

It follows similarly that if $h(p) \in \{q, q'\}$, then $h(P) = Q$ and $h(Q) = P$.
\end{proof}

The following is a lemma which will be referenced in many of the proofs of Section 3. We will denote by $\pi$ the projection of the $xy$-plane onto the $y$-axis so that $\pi(x,y) = (0,y)$ for every $(x,y) \in \R^2$. It should also be noted in the following lemma that by \textit{maximal subarc}, we allow the possibility of the empty set or a singleton.


\begin{lemma}\label{small bends}
Let $(A_1, A_2, A_3, \ldots)$ be a sequence of arcs in $\R^2$ converging to $\{0\} \times [0,1]$ as $i \to \infty$ so that for each $i \in \N$, $\pi \upharpoonright A_i$ is a homeomorphism of $A_i$ onto $\pi(A_i)$. Let $h$ be a homeomorphism of the $xy$-plane onto itself so that $h(\{0\} \times [0,1]) = \{0\} \times [0,1]$. For every $t \in [0,1]$ and for each $i \in \N$, let $C_{t,i}$ denote the maximal subarc of $h(A_i)$ having the property that each of its endpoints lie on the horizontal line given by the equation $y = t$. Then for every $t \in [0,1]$, $\diam(C_{t,i}) \to 0$ as $i \to \infty$.
\end{lemma}
\begin{proof}
Let $t \in [0,1]$, and let $T$ denote the line given by the equation $y = t$. For each $i \in \N$, let $C_i := C_{t,i}$, and let $T_i$ be the line segment contained in $T$ whose endpoints are the endpoints of $C_i$. Note that since $T_i \to \{(0,t)\}$ as $i \to \infty$, we have $h^{-1}(T_i) \to \{h^{-1}((0,t))\}$ as $i \to \infty$.  Thus, the endpoints of $h^{-1}(T_i)$ converge to $\{h^{-1}((0,t))\}$ as $i \to \infty$.  Since for each $i \in \N$, $\pi \upharpoonright A_i$ is a homeomorphism of $A_i$ onto $\pi(A_i)$, and because the endpoints of $h^{-1}(C_i)$ are also the endpoints of $h^{-1}(T_i)$, it follows that $h^{-1}(C_i) \to \{h^{-1}((0,t))\}$ as $i \to \infty$. Therefore $h(C_i) \rightarrow \{(0,t)\}$ as $i \to \infty$, whence $\diam(C_i) \to 0$ as $i \to \infty$.
\end{proof}

\section{Embeddings of $K$: $\aleph_0$-many Endpoints of Layers Inaccessible}

Here, we construct a collection of $\mathfrak{c}$-many mutually inequivalent planar embeddings of $K$. Each embedding from this collection has the property that all but a countably infinite set of layers of continuity (straight-arc layers) have at least one point not being accessible from the complement of the image of the embedding. More precisely, we fix a sequence $\mathcal{Q} = (Q_1, Q_2, Q_3, \ldots )$ of interior layers of continuity of $K$ converging to the left end layer $E$ of $K$, arranged in order from right to left. After this, we choose an arbitrary sequence $(L_1, L_2, L_3, \ldots)$ of $V$ or $\Lambda$ layers of $K$ such that for each $i \in \N$, $L_i$ lies between $Q_i$ and $Q_{i+1}$, with $L_0$ designated as the right end layer $E'$ of $K$. Then, we choose a sequence $A = (a_1, a_2, a_3, \ldots)$ of $0$'s and $1$'s so that for each $i \in \N$, we produce a certain planar re-embedding of $[r(L_i), \ell(L_{i-1})]$ which keeps $Q_i$ straight and perturbs $[r(L_i), Q_i) \cup (Q_i, \ell(L_{i-1})]$ about $Q_i$ in a certain manner depending on when $a_i = 0$ or when $a_i = 1$. In doing so, every point in the image of $Q_i$ under such a re-embedding is inaccessible from the complement when $a_i = 0$, and all put one endpoint of $Q_i$ is inaccessible from the complement if $a_i = 1$. Each such re-embedding is made in such a way that the union of their images together with $E = \{0\} \times [0,1]$ is a re-embedding of $K$. We will see that two embeddings constructed according to different sequences of 0's and 1's will produce inequivalent embeddings, thus providing us with a collection of $\mathfrak{c}$-many mutually inequivalent planar embeddings of $K$.

To understand how $\mathcal{Q}$ is used to provide a collection of mutually inequivalent embeddings, we must motivate how to construct what we call \textit{Type-0} and \textit{Type-1} planar embeddings of $K$ about one of its interior layers of continuity. Suppose $Q$ is an interior layer of continuity of $K$, with its bottom endpoint labeled $b$ and its top endpoint labeled $t$.  Two ways in which we can embed $K$ in the plane is so that the image of $Q$ has none of its points accessible or so that only the image of $b$ is accessible.  Again, let $E$ and $E'$ denote the left and right end layers of $K$ respectively.

We will now describe a decomposition of $[E, Q)$ and $(Q,E']$ as shown in Figure \ref{fig:decomposition}. Let $K_1^{(\ell)}, K_2^{(\ell)}, K_3^{(\ell)}, \ldots$ be a decomposition of $[E,Q)$ with $K_i^{(\ell)} \to Q$ as $i \to \infty$ and so that for each $i \in \N$, 
    \begin{itemize}
        \item[($1_{\ell,i}$.)] $K_i^{(\ell)}$ is a homeomorphic copy of $K$,
        \item[($2_{\ell,i}$.)] $K_i^{(\ell)} \cap K_j^{(\ell)} \neq \emptyset$ if and only if $|i-j| \leq 1$, and 
        \item[($3_{\ell,i}$.)] $K_i^{(\ell)} \cap K_{i+1}^{(\ell)} = \{\ell_i\}$, where $\ell_i$ is the vertex of a $\Lambda$-layer of $K$ if $i$ is odd and is the vertex of a $V$-layer of $K$ if $i$ is even.
    \end{itemize}
Similarly, let $K_1^{(r)}, K_2^{(r)}, K_3^{(r)}, \ldots$ be a decomposition of $(Q,E']$ with $K_i^{(r)} \to Q$ as $i \to \infty$ and so that for each $i \in \N$, 
    \begin{itemize}
        \item[($1_{r,i}$.)] $K_i^{(r)}$ is a homeomorphic copy of $K$,
        \item[($2_{r,i}$.)] $K_i^{(r)} \cap K_j^{(r)} \neq \emptyset$ if and only if $|i-j| \leq 1$, and 
        \item[($3_{r,i}$.)] $K_i^{(r)} \cap K_{i+1}^{(r)} = \{r_i\}$, where $r_i$ is the vertex of a $\Lambda$-layer of $K$ if $i$ is odd and is the vertex of a $V$-layer of $K$ if $i$ is even.
    \end{itemize}
A depiction of the aforementioned decomposition of $[E,Q)$ and $(Q,E']$ is shown in Figure \ref{fig:decomposition}. 

\begin{figure}
\centering
\includegraphics[width=16cm]{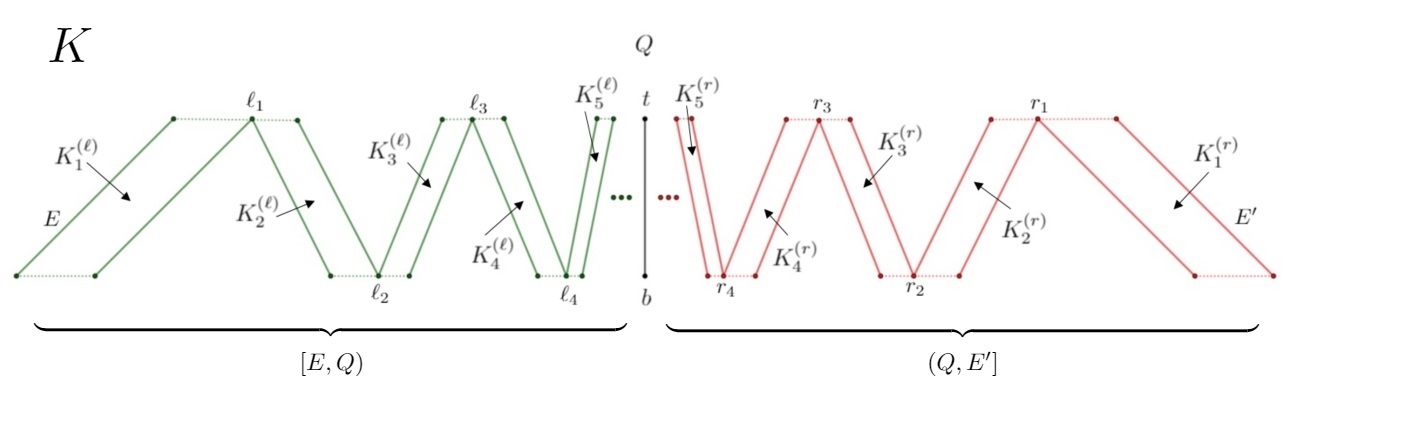}
\caption{The decomposition of $[E, Q)$ and $(Q, E']$, where $E$ and $E'$ are the left and right end-layers of $K$, respectively, and the interior layer of continuity $Q$ is the straight arc in the middle.  The thick red and green line segments other than $E$ and $E'$ are either $\Lambda$-layers of $V$-layers of $K$ defining the aforementioned decompositions, where as the dotted horizontal red and green line segments represent the endpoints and vertices of interior layers of each $K_i^{(\ell)}$ and each $K_i^{(r)}$, respectively.}
\label{fig:decomposition}
\end{figure}

\subsection{Type-$0$ and Type-$1$ Planar Embeddings of $K$ about $Q$}\label{Section Type-0}

\

Given the decompositions in the previous paragraph, we shall describe how to construct a planar embedding $\zeta$ of $K$ so that zero points of $\zeta(Q)$ are accessible from the complement of the $\zeta(K)$. The most efficient way to describe such an embedding is through comparing Figure \ref{fig:decomposition} with Figure \ref{fig:0-type K}. Again, Figure \ref{fig:decomposition} depicts $K$ with the decompositions of $[E, Q)$ and $(Q,E']$ as mentioned above.  Figure \ref{fig:0-type K} depicts $\zeta(K)$ by exhibiting how $Q$ and the elements of the aforementioned decompositions are mapped under $\zeta$.  As shown there, $\zeta(Q)$ is mapped to a vertical line segment so that $\zeta(b)$ is its bottom endpoint and $\zeta(t)$ is its top endpoint.  The rest of $[E, Q)$ and $(Q,E']$ is mapped by $\zeta$ in such a way that the image of the top endpoints and vertices of layers of $K\backslash Q$ converge to $\zeta(t)$ and $\zeta(b)$, respectively,  while bending, stretching, and shrinking members of the decompositions in such a way so that $\zeta: K \to \zeta(K)$ is a homeomorphism. In doing so, no point of $\zeta(Q)$ is accessible from the complement of $\zeta(K)$. Such an embedding is similar in nature to an embedding of a double $\sin(1/x)$-curve with two rays approaching one limiting arc in the middle, as shown in Figure \ref{fig:Type-0}. We call the planar embedding $\zeta$ a \textbf{type-0 planar embedding of $K$ about $Q$}, or just a \textbf{type-0 embedding} for short.

\begin{figure}
\centering
\includegraphics[width=8cm]{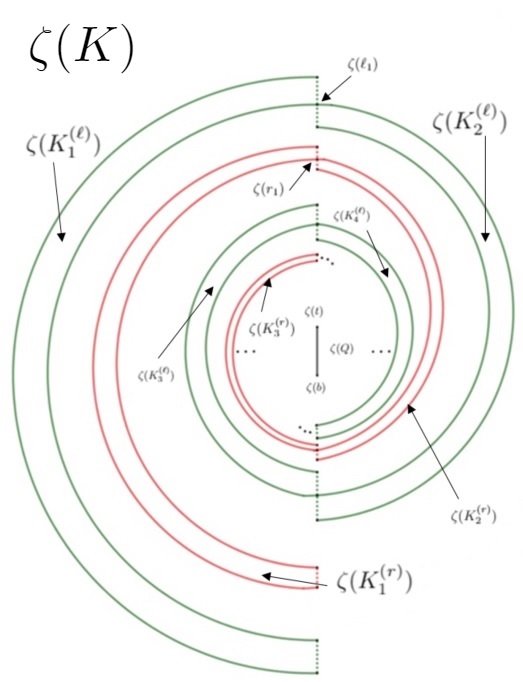}
\caption{A type-0 planar embedding of $K$ about the interior layer of continuity $Q$. Note that $\zeta(Q)$ is completely "buried" by $\zeta([E,Q))$ and $\zeta((Q, E'])$ in the sense that no point of $\zeta(Q)$ is accessible from the complement of $\zeta(K)$.}
\label{fig:0-type K}
\end{figure}


\begin{figure}
\centering
\includegraphics[width=13cm]{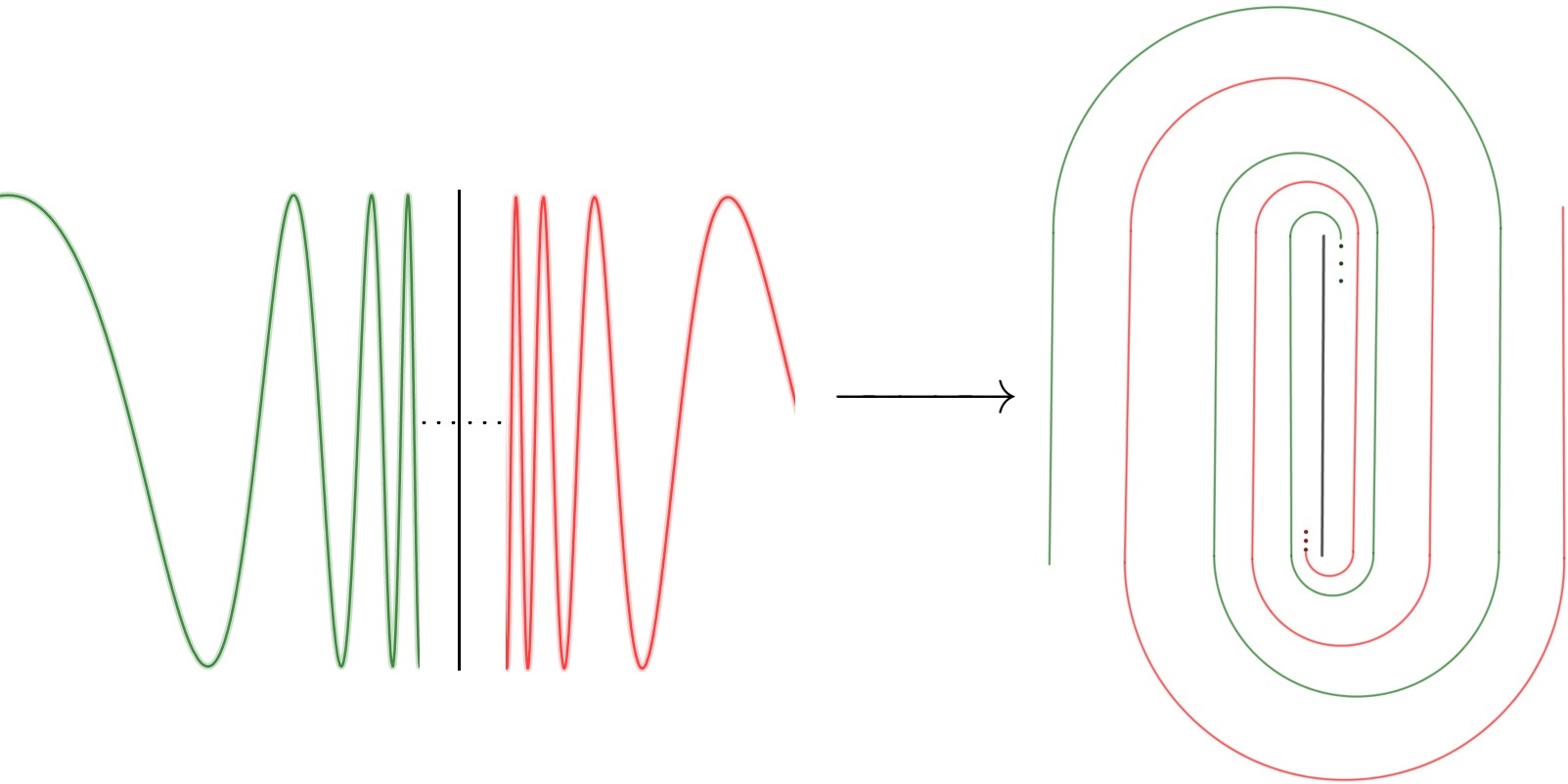}
\caption{Here, we have a double $\sin(1/x)$-curve with two rays on opposite sides of a single limiting arc being re-embedded in the plane so that both approaching rays completely "bury" the limiting arc, making no point of it accessible from the complement.  This embedding stands as a model for the type-0 planar embedding $\zeta$ of $K$ shown in Figure \ref{fig:0-type K}.}
\label{fig:Type-0}
\end{figure}

Also given the previous decompositions of $[E,Q)$ and $(Q, E']$, we shall next describe how to construct a planar embedding $\xi$ of $K$ so that only one of the points of $\xi(Q)$ is accessible from the complement of $\xi(K)$.  In fact, since the interiors of interior layers of continuity of $K$ are inaccessible under any planar embedding of $K$, the single accessible point under this type of embedding will be an endpoint of $\xi(Q)$. Just as in the description of the type-0 embedding, the most efficient way to describe the embedding $\xi$ is through comparing Figure \ref{fig:decomposition} with Figure \ref{fig:1-type K}. Figure \ref{fig:1-type K} depicts $\xi(K)$ by showing how $Q$ and the elements of the aforementioned decompositions are mapped under $\xi$.  As shown there, $\xi(Q)$ is mapped to a vertical line segment so that $\xi(b)$ is its bottom endpoint and $\xi(t)$ is its top endpoint.  The rest of $[E, Q)$ and $(Q,E']$ is mapped by $\xi$ in such a way that the image of the top endpoints and vertices of layers of $K\backslash Q$ converge to $\xi(t)$ and $\xi(b)$, respectively, while bending, stretching, and shrinking members of the decompositions in such a way so that $\xi: K \to \xi(K)$ is a homeomorphism. In doing so, $\xi(b)$ is accessible from the complement of $\xi(K)$ while every other point of $\xi(Q)$ is not. Such an embedding is similar in nature to an embedding of a double $\sin(1/x)$-curve with two rays approaching one limiting arc in the middle, as shown in Figure \ref{fig:Type-1}. We call the planar embedding $\xi$ a \textbf{type-1 planar embedding of $K$ about $Q$}, or just a \textbf{type-1 embedding} for short. 

\begin{proposition}\label{zeta not xi}
Let $\zeta$ and $\xi$ be a type-0 and type-1 planar embedding of $K$ about an interior layer of continuity, $Q$, of $K$.  Then $\zeta$ and $\xi$ are inequivalent planar embeddings of $K$.
\end{proposition}
\begin{proof}
Since $\zeta(K)$ and $\xi(K)$ have a different set of accessible points, in particular, since no endpoints of $\zeta(Q)$ is accessible from the complement of $\zeta(K)$ while one endpoint of $\xi(Q)$ is accessible from the complement of $\xi(K)$, it follows by Proposition \ref{different access} that $\zeta$ and $\xi$ are inequivalent planar embeddings of $K$.
\end{proof}


\begin{figure}
\centering
\includegraphics[width=13cm]{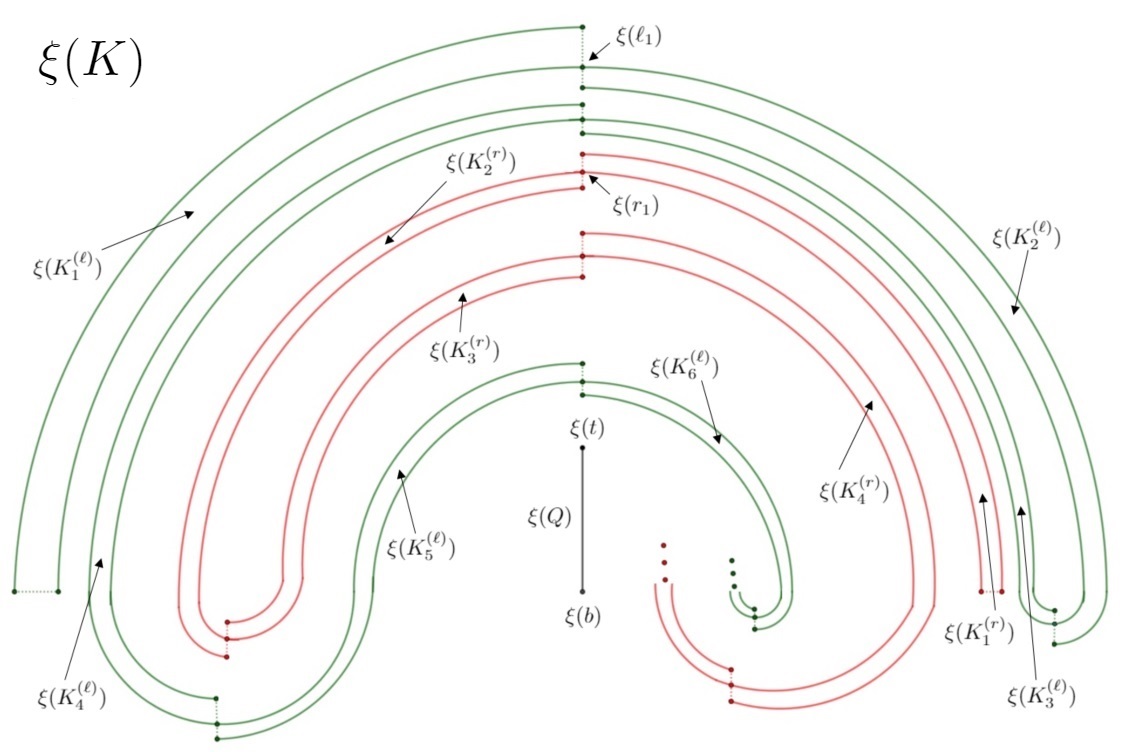}
\caption{A type-1 planar embedding of $K$ about the interior layer of continuity $Q$.  Note that $\xi(b)$ will remain accessible from the complement of $\xi(K)$, but $\xi(t)$ will not.}
\label{fig:1-type K}
\end{figure}

\begin{figure}
\centering
\includegraphics[width=13cm]{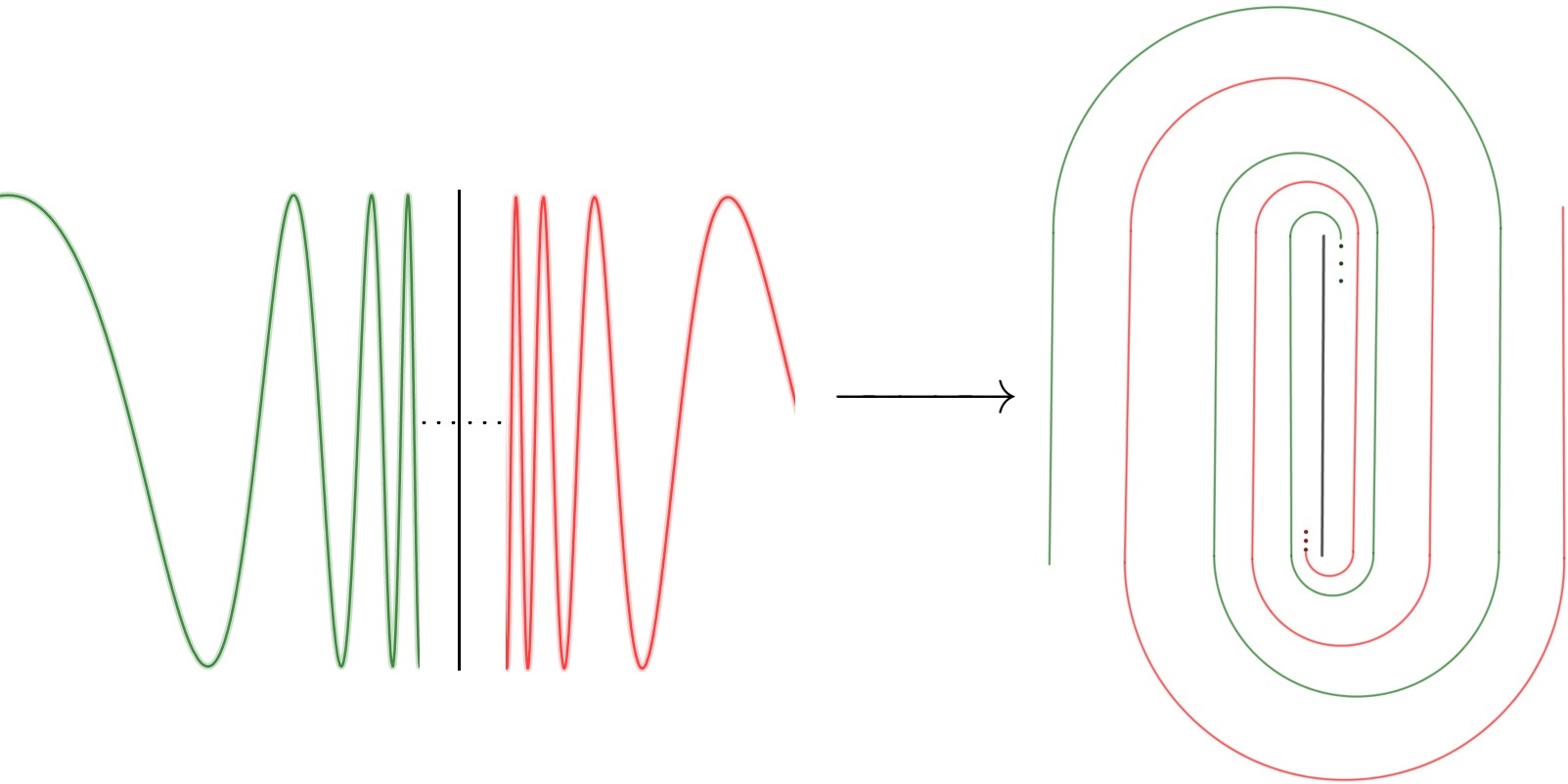}
\caption{As in Figure \ref{fig:Type-0}, we have a double $\sin(1/x)$-curve with two rays on opposite sides of a single limiting arc being re-embedded in the plane, this time so that both limiting arcs "bury" the top endpoint of the limiting arc, leaving only one point---the bottom endpoint---accessible from the complement.  This embedding stands as a model for the type-1 planar embedding $\xi$ of $K$ shown in Figure \ref{fig:1-type K}.}
\label{fig:Type-1}
\end{figure}

\subsection{A Collection of $\mathfrak{c}$-many Mutually Inequivalent Planar Embeddings of $K$}

\

We will now construct the $\mathfrak{c}$-many mutually inequivalent planar embeddings of $K$, all of which do not preserve the accessibility of points in $K$ which are accessible from the complement of the standard embedding of $K$. For each such embedding, we will make use of a fixed sequence $\mathcal{Q} = (Q_1, Q_2, Q_3, \ldots)$ of interior layers of continuity of $K$ converging to the left end layer $E$ of $K$ as mentioned in the first paragraph of this section.  For convenience, we may assume that the $Q_i$'s are in order from right to left. That is, if $g$ is a Kuratowski map of $K$ such that $g(E) = \{0\}$, then $g(Q_{i+1}) < g(Q_i)$ for each $i \in \N$.

For each $i \in \N$, let $L_i$ be either a $V$-layer or $\Lambda$-Layer of $K$ in between $Q_i$ and $Q_{i+1}$, and designate $L_0$ to be the right end-layer, $E'$, of $K$.  Also for each $i \in \N$, decompose $[r(L_i), Q_i)$ and $(Q_i, \ell(L_{i-1})]$ in the same way as the decomposition of $[E,Q)$ and $(Q,E']$ as given at the beginning of this section.

Let $A = (a_1, a_2, a_3, \ldots)$ be a sequence so that for each $i \in \N$, $a_i = 0$ or $a_i = 1$.  That is, $A$ is a sequence of 0's and 1's.  For each $i \in \N$, let $a_i$ be assigned to the layer $Q_i$.  If $a_i = 0$, replace $[r(L_i), \ell(L_{i-1})]$ (which is homeomorphic to $K$) with a type-0 embedding of $[r(L_i), \ell(L_{i-1})]$ about $Q_i$ by using the aforementioned decompositions of $[r(L_i), Q_i)$ and $(Q_i, \ell(L_{i-1})]$.  If $a_1 = 1$, replace $[r(L_i), \ell(L_{i-1})]$ with a type-1 embedding of $[r(L_i), \ell(L_{i-1})]$ about $Q_i$, again by using the aforementioned decompositions of $[r(L_i), Q_i)$ and $(Q_i, \ell(L_{i-1})]$.  Furthermore, make all such replacements be so that their images converge to $E$ as $i \to \infty$, resulting in the image of a planar embedding of $K$.  We will call such an embedding a \textbf{type-$A$ planar embedding of $K$ about $\mathcal{Q}$}, or just a \textbf{type-$A$ embedding} for short.

Let $\mathcal{Z}$ be the collection of all sequences of 0's and 1's.

\begin{lemma}
Let $A = (a_1, a_2, a_3, \ldots)$ and $B = (b_1, b_2, b_3, \ldots)$ be nonidentical sequences in $\mathcal{Z}$ and let $\alpha$ and $\beta$ denote type-$A$ and type-$B$ embeddings of $K$ about $\mathcal{Q}$, respectively.  Then $\alpha$ and $\beta$ are inequivalent planar embeddings of $K$. 
\end{lemma}\label{inequivalent embeddings weak}
\begin{proof}
Suppose $h$ is a homeomorphism of the plane onto itself so that $h(\alpha(K)) = \beta(K)$.  By Corollary \ref{like layers}, it follows that $h(\alpha(Q_i))$ is a layer of continuity of $\beta(K)$ for each $i \in \N$.  Since for each $i \in \N$, $\alpha(Q_i)$ has all but at least one endpoint inaccessible from the complement of $\alpha(K)$, it follows that $h(\alpha(Q_i))$ must be be a $\beta(Q_{j(i)})$ for some $j(i) \in \N$.  Furthermore, by Proposition \ref{preserve order}, we have $j(i) = i$ for each $i \in \N$.  Since $A$ and $B$ are nonidentical, there is a $k \in \N$ such that $a_k \neq b_k$, in which case $h(\alpha(Q_k)) = \beta(Q_k)$ and thus, $h(\alpha([r(L_k),\ell(L_{k-1})])) = \beta([r(L_k),\ell(L_{k-1})])$.  However, this contradicts Proposition \ref{zeta not xi} since $\alpha\upharpoonright{[r(L_k),\ell(L_{k-1})]}$ is a type-0 embedding about $Q_k$ and $\beta\upharpoonright{[r(L_k),\ell(L_{k-1})]}$ is a type-1 embedding about $Q_k$.
\end{proof}

\begin{theorem}\label{c many embeddings weak}
There exist $\mathfrak{c}$-many mutually inequivalent planar embeddings of $K$.
\end{theorem}
\begin{proof}
Let $\mathcal{E}$ be the collection of $A$-type embeddings of $K$ about $\mathcal{Q}$ for each $A \in \mathcal{Z}$.  By Lemma \ref{inequivalent embeddings weak}, members of $\mathcal{E}$ are pairwise inequivalent.  Therefore, $|\mathcal{E}| = |\mathcal{Z}|$.  Since it is well known that $|\mathcal{Z}| = \mathfrak{c}$, it follows that $|\mathcal{E}| = \mathfrak{c}$.
\end{proof}

Theorem \ref{c many embeddings weak} gives greater insight to Question 6 in \cite{Anusic2}, providing an example of an HDCC having uncountably many, and, in fact, $\mathfrak{c}$-many mutually inequivalent planar embeddings which is not an HDCC containing a dense ray.  Furthermore, $K$ is an HDCC satisfying this property while having no subcontinuum containing a dense ray.  However, these embeddings fail to preserve the accessibility of all points which are accessible in the image of the standard planar embedding of $K$.

\section{Embeddings of $K$: Endpoints of All Layers Accessible}

In this section, we construct a collection of $\mathfrak{c}$-many mutually inequivalent planar embeddings of $K$, each of whose image has the same set of accessible points as the image of the standard embedding of $K$. That is, under each such embedding, the image of every point of each end layer, $V$-layer, and $\Lambda$-layer will be accessible, and the image of both endpoints of each interior layer of continuity will also be accessible. Before proceeding, we must first provide the following definition and lemma.  

\begin{definition}\label{inequivalent sequences}
Let $\mathbf{N}$ and $\mathbf{M}$ be sequences of positive integers.  We say that $\mathbf{N}$ and $\mathbf{M}$ are inequivalent if and only if, after removing any finite initial subsequence of $\mathbf{N}$ and any finite initial subsequence of $\mathbf{M}$, the remaining sequences $\mathbf{N}'$ and $\mathbf{M}'$ are not identical.
\end{definition}

\begin{lemma}\label{uncountable sequences}
There exist $\mathfrak{c}$-many mutually inequivalent sequences of positive integers.
\end{lemma}
\begin{proof}
Let $\mathcal{N}$ denote the set of all sequences of positive integers, and let $A = (a_1, a_2, a_3, \ldots) \in \mathcal{N}$.  We shall inductively define the sets $A_n$ so that $A_1 = \{A\}$, and for every integer $n > 1$, 
$$A_n = \{(s_1, s_2, s_3, \ldots) \in \mathcal{N} \mid (s_n, s_{n+1}, s_{n+2}, \ldots) = (a_n, a_{n+1}, a_{n+2}, \ldots)\}.$$  
Note that $A_n$ is countable for every $n =  1, 2, \ldots$. Let $\mathcal{A} = \bigcup_{n = 1}^\infty A_n$, which forms an equivalence class of all sequences of positive integers equivalent to $A$. Then $\mathcal{A}$ is also countable as it is the countable union of countable sets.  Denote by $\mathcal{I}$ the set of all such equivalence classes, $\mathcal{A}$.

Since each equivalence class in $\mathcal{I}$ is countable and $\bigcup \mathcal{I} = \mathcal{N}$ which has cardinality $\mathfrak{c}$, it follows that the cardinality of $\mathcal{I}$ is a cardinal $\kappa$ satisfying $\mathfrak{c} = \kappa \otimes \aleph_0$, where $\otimes$ here denotes cardinal multiplication.  By Corollary 10.13 of \cite{Kunen}, as well as Section IX.6 of \cite{Sierpinski}, it follows that $\kappa = \mathfrak{c}$.
\end{proof}

\subsection{Constructing Schema Embeddings of $K$}

\

Let $\mathbf{N} = (n_1, n_2, n_3, \ldots)$ be a sequence of positive even integers greater than or equal to 4. We will construct a planar embedding $\psi_\mathbf{N}$ of $K$ based on a subsequently defined set of instructions for geometrically altering parts of the standard embedding of $K$. To do so, we must first define a decomposition of each $K_i$, an example of which is shown in Figure \ref{fig:8 pieces}. For each $i \in \N$, let $K_i \subset K$ be a homeomorphic copy of $K$ so that the right end layer of $K_1$ is the right end layer of $K$, so that $K_i \cap K_j \neq \emptyset$ if and only if $|i-j| \leq 1$, and so that $K_i \cap K_{i+1} = \{p_i\}$ is a vertex of a $V$-layer of $K$.  Also for each $i \in \N$, let $K_i^{(1)}, \ldots, K_i^{(2n_i)}$ be a decomposition of $K_i$ so that 
    \begin{itemize}
        \item[($1_{K_i}$.)] $K_i^{(l)}$ is a homeomorphic copy of $K$,
        \item[($2_{K_i}$.)] the right end layer of $K_i^{(1)}$ is the right end layer of $K_i$ and the left end layer of $K_i^{(2n_i)}$ is the left end layer of $K_i$,
        \item[($3_{K_i}$.)] $K_i^{(l)} \cap K_i^{(m)} \neq \emptyset$ if and only if $|l-m| \leq 1$,
        \item[($4_{K_i}$.)] $K_i^{(l)} \cap K_i^{(l+1)} = \{s_i^{(l)}\}$ is the vertex of a $\Lambda$-layer of $K$ when $l$ is odd and is the vertex of a $V$-layer when $l$ is even.
    \end{itemize}

\begin{figure}
\centering
\includegraphics[width=15cm]{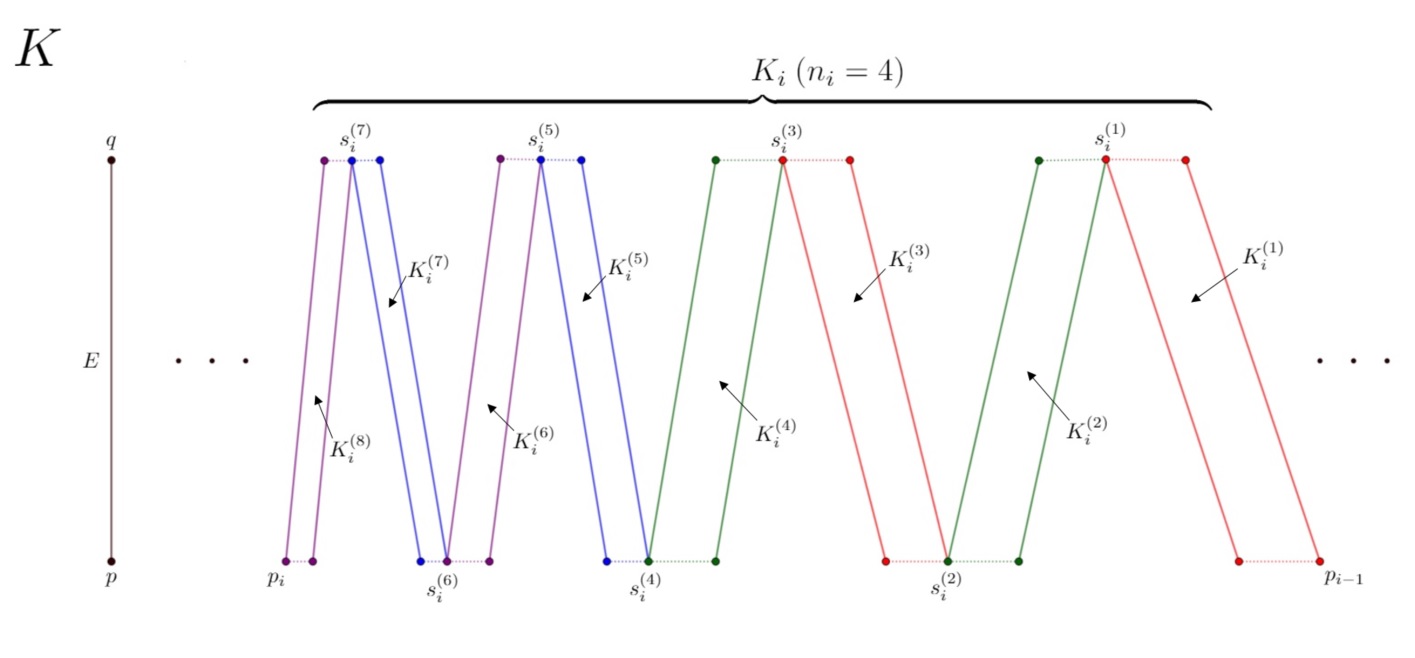}
\caption{$K_i$, where $n_i = 4$, decomposed into $2n_i = 8$ pieces each homeomorphic to $K$.}
\label{fig:8 pieces}
\end{figure}


Again, the above decomposition of $K_i$ is depicted in Figure \ref{fig:8 pieces}.
Recall that in section 1, we let $P$ denote the set of all points in the standard embedding of $K$ whose $y$-coordinates are 0, and we let $Q$ denote the set of all points in the standard embedding of $K$ whose $y$-coordinates are 1. That is, $P$ is the set of all bottom endpoints and vertices of layers of $K$ while $Q$ is the set of all top endpoints and vertices of layers of $K$. Each of the next constructed embeddings will have the property that $P$ is mapped below the line $y = 1/2$ and that $Q$ is mapped above the line $y = 1/2$. We shall thus refer to the line $y = 1/2$ as \textit{the critical line} of each of the following embeddings.

We will inductively define, for each $i \in \N$, a list of planar embeddings of $K_i$, whose images are resembled in Figure \ref{fig:subschema}, as follows.  First, embed $K_1$ in the $xy$-plane where $x > 0$ so that the endpoints and vertices of layers $K_1$ are contained outside of the critical line, with the top endpoints and vertices of $K_1$ in the part above the critical line, and the bottom endpoints and vertices in the part below the critical line, with every arc that is straight in $K_1$ kept straight under this embedding.  We then change this embedding as follows.  Reflect, through ambient three-dimensional space, $K_1^{(n_1)} \cup \cdots \cup K_1^{(2n_1)}$, about the point $s_1^{(n_1-1)}$ where $K_1^{(n_1 - 1)}$ and $K_1^{(n_1)}$ intersect, in such a way that the point where $s_1^{(n_1 + k)}$
is vertically collinear with $s_1^{(n_1 - k)}$ for every $k = 1, 2, \ldots, n_1 - 1$.\footnote{By \textit{vertically collinear}, we mean that these points lie on the same vertical line in the $xy$-plane.} Moreover, this is done while keeping top and bottom endpoints and vertices of layers of $K_1$ respectively above and below the critical line and keeping straight under this re-embedding of $K_1$ every arc that is straight in the standard embedding of $K_1$, smoothly stretching and squeezing where needed.
We shall let this re-embedding of $K_1$ be named $\psi_{\mathbf{N},1}$.

Assume now that for some $i \in \N$, we have constructed $\psi_{\mathbf{N},j}$ for every $j = 1, \ldots, i$. We construct $\psi_{\mathbf{N},i+1}$ somewhat similar to the way $\psi_{\mathbf{N},1}$ was constructed, this time so that $\psi_{\mathbf{N},i+1}(K_{i+1})$ meets with $\psi_{\mathbf{N},i}(K_i)$ at only $\psi_{\mathbf{N},i}(p_i)$ with $K_{i+1}^{(1)}$ being smoothly stretched by $\psi_{\mathbf{N},i+1}$ below and to the left of $\psi_{\mathbf{N},i}(K_i)$, with the rest of $\psi_{\mathbf{N},i+1}(K_{i+1})$ placed between the $y$-axis and $\psi_{\mathbf{N},i+1}(K_{i+1}^{(1)}) \cup \bigcup_{j=1}^i \psi_{\mathbf{N},j}(K_j)$.  In doing so, the only arcs that are straight in the standard embedding of $K_{i+1}$ being kept straight under $\psi_{\mathbf{N},i+1}$ are those contained in $K_{i+1}^{(2)}, \ldots, K_{i+1}^{(2n_{i+1})}$. Furthermore, we make sure that $\psi_{\mathbf{N},i}(K_i)$ converges to the line segment $\{0\} \times [0,1]$ as $i \to \infty$, doing so in a way that the resulting function $\psi_\mathbf{N}: K \to \psi_\mathbf{N}(K)$, as given in the following definition, is a homeomorphism, and thus, a planar embedding of $K$.

\begin{definition}\label{schema Embedding}
Given the elements in the construction above, the \textbf{schema embedding of $K$ according to $\mathbf{N}$}, denoted as $\psi_\mathbf{N}$, is the planar embedding of $K$ whose image is given by
$$\psi_\mathbf{N}(K) = \cl\Big(\bigcup_{i=1}^\infty \psi_{\mathbf{N},i}(K_i)\Big) = (\{0\} \times [0,1]) \cup \bigcup_{i=1}^\infty \psi_{\mathbf{N},i}(K_i),$$
where the left end layer $E$ of $K$ is mapped onto the line segment $\{0\} \times [0,1]$ so that the bottom endpoint $p$ of $E$ is mapped to $(0,0)$ and the top endpoint $q$ of $E$ is mapped to $(0,1)$. That is, $\psi_\mathbf{N}$ is defined by 
\[ \psi_\mathbf{N}(x) = \begin{cases} 
          \psi_{\mathbf{N},i}(x) & \textnormal{if } x \in K_i \\
          x & \textnormal{if } x \in E
       \end{cases}
    \]
\end{definition}

A depiction showing $\psi_\mathbf{N}(K)$ with $\psi_{\mathbf{N},i}(K_i)$, where $n_i = 4$, is given in Figure \ref{fig:subschema}. Such embeddings somewhat mimic the type of embeddings of the $\sin(1/x)$-curve portrayed in Figure \ref{fig:redtopsine}. It is worth noting that although $\psi_\mathbf{N}(K_i^{(1)})$ is bent for each $i \in \N$, these bends become less profound as $i$ increases.  More precisely, for every $\epsilon > 0$ there exists an $N_\epsilon \in \N$ such that for all $i \geq N_\epsilon$ and for every maximally straight arc $A$ contained in the standard embedding of $K_i^{(1)}$, there exists a homeomorphism $h_A$ mapping $\psi_\mathbf{N}(A)$ onto $\pi(\psi_\mathbf{N}(A))$ such that $\|(\pi\upharpoonright\psi_\mathbf{N}(A)) - h_A\| < \epsilon$, where $\| \cdot \|$ is the supremum norm.  Furthermore, recall again that we denoted by $P$ and $Q$ the set of all bottom and top endpoints and vertices of layers of $K$, respectively. Likewise, for each $i \in \N$, denote by $P_i$ and $Q_i$ the set of all bottom and top endpoints and vertices of layers of $K_i$, respectively. We also note that in order to make $\psi_\mathbf{N}$ a homeomorphism onto its image, we ensure that $\psi_\mathbf{N}(P_i) \to \{\psi_\mathbf{N}(p)\} = \{(0,0)\}$ and $\psi_\mathbf{N}(Q_i) \to \{\psi_\mathbf{N}(q)\} = \{(0,1)\}$ as $i \to \infty$.

\begin{proposition}
For every sequence $\mathbf{N}$ of positive even integers greater than or equal to 4, $\psi_\mathbf{N}$ is a homeomorphism of $K$ onto its image. That is, $\psi_\mathbf{N}$ is indeed a planar embedding of $K$.
\end{proposition}

\begin{figure}
\centering
\includegraphics[width=15cm]{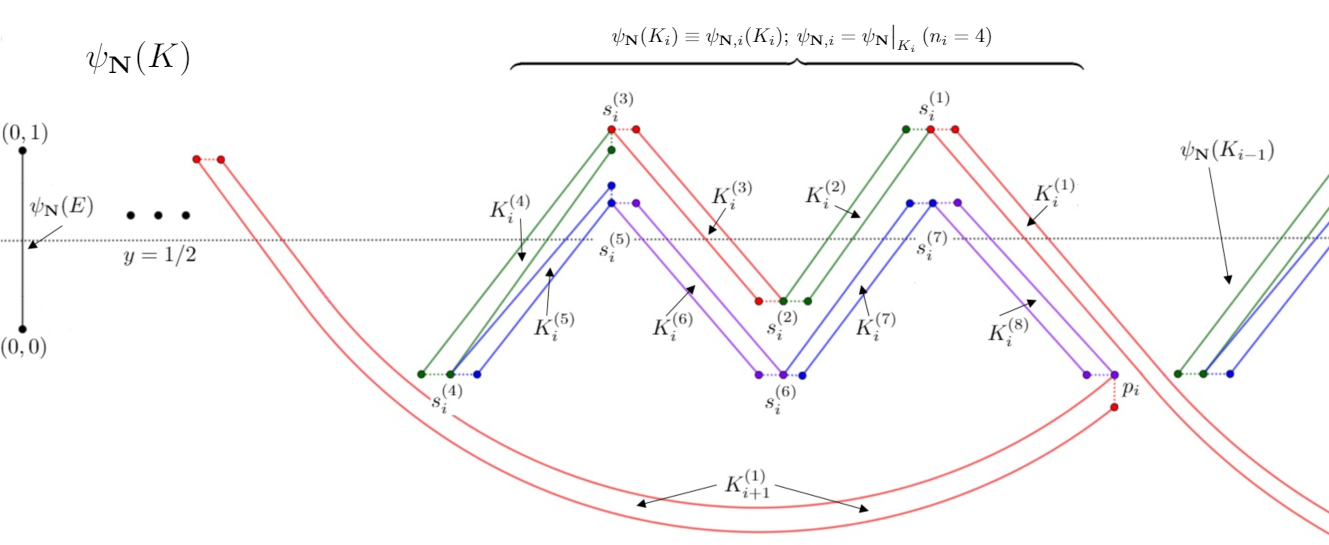}
\caption{A subschema embedding, $\psi_{\mathbf{N},i}(K_i)$, with $n_i = 4$.  Each labeled $K_i^{(l)}$ is really $\psi_{\mathbf{N}}(K_i^{(l)})$ for each $l = 1, 2, \ldots, 8$. Note that we have $\psi_{\mathbf{N},i+1}(K_{i+1})$ meeting $\psi_{\mathbf{N},i}(K_i)$ at the point $\psi_{\mathbf{N},i}(p_i) = \psi_{\mathbf{N},i+1}(p_i)$, simply labeled $p_i$ on the figure.}
\label{fig:subschema}
\end{figure}

\begin{figure}
\centering
\includegraphics[width=8cm]{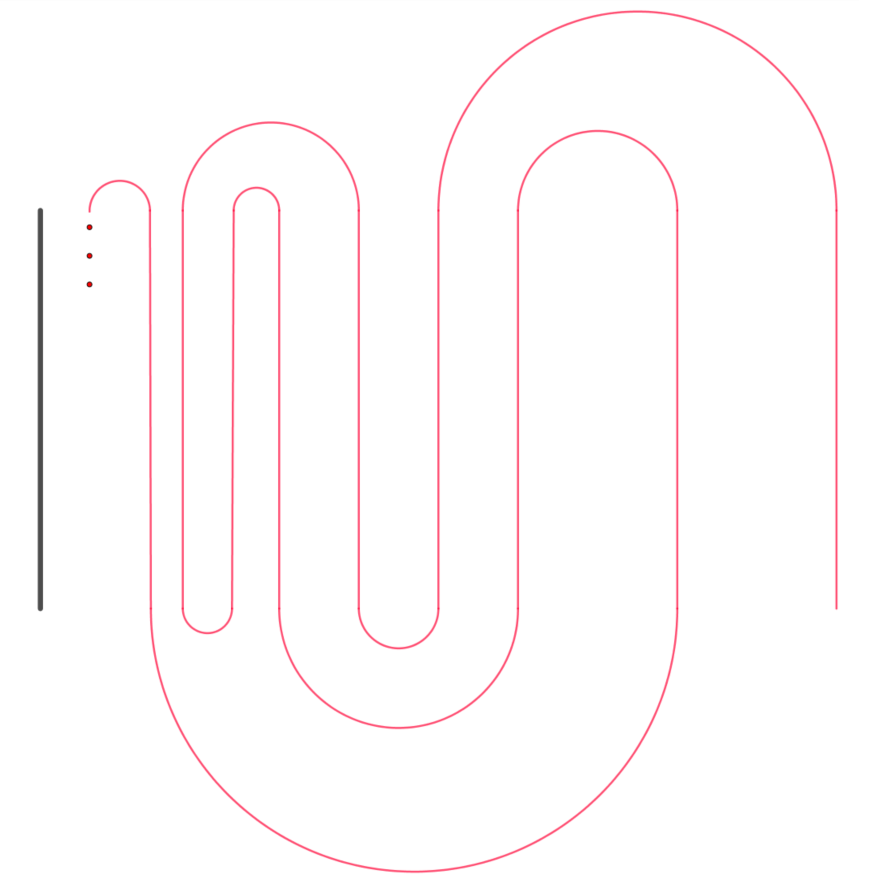}
\caption{The image of an embedding of the $\sin(1/x)$-curve which stands as a crude model for the image of $K$ under a schema embedding $\psi_\mathbf{N}$ according to a sequence $\mathbf{N}$ of positive even integers greater than or equal to 4, with $n_1 =4$.}
\label{fig:redtopsine}
\end{figure}

\begin{remark}\label{remark on n}
If $\mathbf{N} = (n, n, n, \ldots)$ is a sequence of which every term is the same even positive integer, $n$, greater than or equal to 4, then we may denote the schema embedding $\psi_\mathbf{N}$ as $\psi_\mathbf{n}$.  Thus, it should be understood what is meant by, say, the embeddings $\psi_\mathbf{4}$, $\psi_\mathbf{6}$, $\psi_\mathbf{8}$, etc.
\end{remark}

We allow intuition based on the geometric description of the constructions of each $\psi_\mathbf{N}(K)$ to substitute as sketches of proofs for the next two propositions.

\begin{proposition}\label{layers accessible}
Given a sequence $\mathbf{N}$ of positive even integers greater than or equal to 4, all endpoints of layers of $K$ and all points of $V$ and $\Lambda$ layers of $K$ are accessible under $\psi_\mathbf{N}$.
\end{proposition}

That is, Proposition \ref{layers accessible} states that all points which are accessible from the complement of the standard embedding of $K$ are mapped by $\psi_\mathbf{N}$ so that they are accessible from the complement of $\psi_\mathbf{N}(K)$ for any given sequence $\mathbf{N}$ of positive even integers greater than or equal to 4.

\begin{proposition}\label{equiv seq equiv emb}
If $\mathbf{N}$ and $\mathbf{M}$ are equivalent sequences of positive even integers greater than or equal to 4, then $\psi_\mathbf{N}$ and $\psi_\mathbf{M}$ are equivalent planar embeddings of $K$.
\end{proposition}

\subsection{Schema Pockets and Escape Arcs}

\

In the remainder of this subsection, unless otherwise specified, it should be understood that $\mathbf{N}$ (likewise, $\mathbf{M}$) is any sequence of positive integers greater than or equal to 4. For every $i \in \N$, there is a horizontal crosscut $C_{\mathbf{N},i}$ of $\psi_{\mathbf{N},i}(K_i)$ below the critical line with one endpoint being $\psi_\mathbf{N}(p_i)$ and other endpoint on the left end layer of $\psi_\mathbf{N}(K_i^{(1)})$. This means $\psi_{\mathbf{N}}(K) \cup C_{\mathbf{N},i}$ separates the plane into two open connected components, the bounded component being a topological open disk which we will call the \textbf{$i^\textnormal{th}$ schema pocket} of the complement of $\psi_\mathbf{N}(K)$ which we denote by $D_{\mathbf{N},i}$. (See Figure \ref{fig:escape}.)

For each $i \in \N$, let $q_i$ be the top endpoint of $r(K_i)$. Then there is also a straight crosscut $C_{\mathbf{N},i}'$ having one endpoint being $\psi_\mathbf{N}(q_{i+1})$ and the other being the top endpoint of the left end layer of $\psi_\mathbf{N}(K_i^{(n_i)})$, with $C_{\mathbf{N},i}' \rightarrow (0,1)$ as $i \rightarrow \infty$. This means $\psi_\mathbf{N}(K) \cup C_{\mathbf{N},i}'$ separates the plane into two open connected components, the bounded component also being a topological open disk which we will call the \textbf{$i^\textnormal{th}$ alternate schema pocket} of $\psi_\mathbf{N}(K)$ which will be denoted by $D_{\mathbf{N},i}'$. (Again, see Figure \ref{fig:escape}.)

Let $L$ be a $V$-layer or $\Lambda$-layer of $K$.  Then there exists a straight crosscut $C$ whose endpoints are the endpoints of $\psi_\mathbf{N}(L)$.
Thus, $\psi_\mathbf{N}(L) \cup C$ separates the plane into two open connected components, the bounded component being a topological open disk which we will call the \textbf{$\psi_\mathbf{N}(L)$-pocket}, or simply the \textbf{$L$-pocket} when the embedding $\psi_\mathbf{N}$ is understood. Such a pocket may be denoted as $D_{\mathbf{N},L}$.


\begin{remark}\label{remark on 1}
We may designate the standard embedding of $K$ as $\psi_\mathbf{1}$, where $\mathbf{1} = (1, 1, 1, \ldots)$. Therefore, the $i^\textnormal{th}$ schema pocket, $D_{\mathbf{1},i}$, is not unique, but can be chosen to be a $V$-pocket or $\Lambda$-pocket of the standard embedding of $K$. Furthermore, suppose $L_i$ is the the $V$-layer or $\Lambda$-layer of $\psi_\mathbf{1}(K)$ contained in the boundary of the pocket $D_{\mathbf{1},i}$. Then if $g$ is a Kuratowski map of $K$, $L_i$ can be chosen so that $L_i \rightarrow E$ as $i \rightarrow \infty$, with $g(L_{i+1}) < g(L_i)$  for each $i \in \N$.
\end{remark}

\begin{definition}\label{escape arc}
Let $x \in \R^2\backslash \psi_\mathbf{N}(K)$. Let $J$ be an arc contained in $\R^2\backslash \psi_\mathbf{N}(K)$ having $x$ as an endpoint such that its other endpoint $x'$ is contained in neither a schema pocket nor a subschema pocket of the complement of $\psi_\mathbf{N}(K)$.\footnote{The purpose for the condition on the positioning of the other endpoint of $J$ will become evident when understanding the notion of the depth given in Definition \ref{depth}.} If also there exists a straight vertical ray $A$ in the complement of $\psi_\mathbf{N}(K)$ having $x'$ as its endpoint, then we say $J$ is an \textbf{escape arc of $x$ with respect to $\mathbf{N}$}. We may say that $J$ is an escape arc of $x$, or simply, that $J$ is an escape arc when the initial endpoint $x$ of $J$ and the embedding $\psi_\mathbf{N}$ are understood.
\end{definition}

That is, $J$ is an escape arc of $x$ with respect to $\mathbf{N}$ if $J$ is a path which $x$ may follow so that it may become "free to escape" arbitrarily far away below $\psi_\mathbf{N}(K)$ or arbitrarily far away above $\psi_\mathbf{N}(K)$ via a strait vertical ray $A$ appended to $J$ at its other endpoint $x'$.

\begin{definition}\label{minimal escape arc}
Let $J$ be as in Definition \ref{escape arc}. Suppose further that $J$ can be decomposed into $J_1, \ldots, J_n$ such that for each $i = 1, \ldots, n$, $J_i$ is a maximal subarc of $J$ having the property that $\pi \upharpoonright J_i$ is a homeomorphism onto its image. If $n$ is the least possible integer satisfying such a condition, then we say that $J$ is a \textbf{minimal escape arc of $x$ with respect to $\mathbf{N}$}. We may say that $J$ is a minimal escape arc of $x$, or simply, that $J$ is a minimal escape arc when the initial endpoint $x$ of $J$ and the embedding $\psi_\mathbf{N}$ are understood.
\end{definition} 

The meaning of Definition \ref{minimal escape arc} is that a minimal escape arc is an escape arc that does not take random and unnecessary turns to escape below or above $\psi_\mathbf{N}(K)$. See Figure \ref{fig:escape} for examples of minimal escape arcs.

\begin{figure}
\centering
\includegraphics[width=15cm]{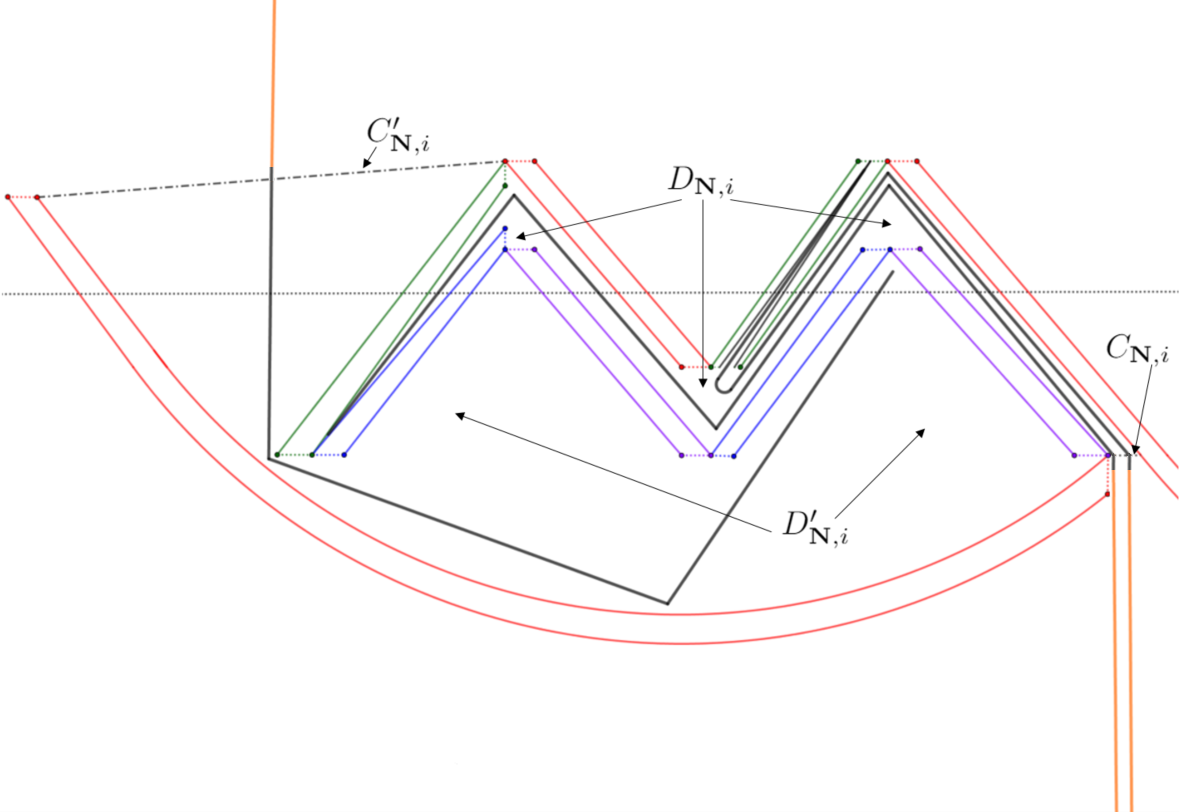}
\caption{Depicted here is $\psi_\mathbf{N}(K_i) \cup \psi_\mathbf{N}(K_{i+1}^{(1)})$ with crosscuts $C_{\mathbf{N},i}$ and $C_{\mathbf{N},i}'$ and their corresponding pockets $D_{\mathbf{N},i}$ and $D_{\mathbf{N},i}'$, respectively. What is also shown is an escape arc out of $D_{\mathbf{N},i}$ having depth 4, and escape arc out of $D_{\mathbf{N},i}$ from a $\Lambda$-pocket having a depth of 3, and an escape arc out of $D_{\mathbf{N},i}'$ having a depth of 2. The orange segments depict vertical rays moving away from the escape arcs.
}
\label{fig:escape}
\end{figure}

\begin{remark}\label{alternate escapes}
Given a schema embedding $\psi_\mathbf{N}$ and a point $x$ in the complement of $\psi_\mathbf{N}(K)$, we may wish to specify the type of pocket $x$ may escape from through an escape arc $J$. That is, suppose $R$ is a region which is either a schema pocket, an alternate schema pocket, a $V$-pocket, or a $\Lambda$-pocket of the complement of $\psi_\mathbf{N}(K)$.  If $x \in R$ and $J$ an escape arc of $x$, then we may say that $J$ is an escape arc of $x$ from $R$.
\end{remark}

\begin{definition}\label{turning points}
Let $J$ be a minimal escape arc from a point of a schema pocket of the complement of a schema embedding of $K$, and let $z \in J$.  We say that $z$ is a \textbf{top (resp., bottom) turning point of $J$} if there exists an $\epsilon > 0$ so that for each horizontal line segment $H$ contained in the open ball $B_\epsilon(z)$ of radius $\epsilon$ centered on $z$ with the boundary of $H$ contained in the boundary of $B_\epsilon(z)$, $H \cap J = \{z\}$ if and only if $H$ contains $z$, $H \cap J$ contains two points if $H$ is below (resp., above) $z$, and $H \cap J = \emptyset$ if $H$ is above (resp., below) $z$.
\end{definition}

That is, $z$ is a turning point of an escape arc $J$ if $\{z\} = J_i \cap J_{i+1}$ for some $i = 1, \ldots, n-1$, where $J_1, \ldots, J_n$ is the decomposition of $J$ into the least number of subarcs having the property that $\pi \upharpoonright J_i: J_i \rightarrow \pi(J_i)$ is a homeomorphism.

We now introduce the notion of depth of points (and subsets) in the complement of a given schema embedding.

\begin{definition}\label{depth}
Let $x$ be as in Definitions \ref{escape arc} and \ref{minimal escape arc}.  The \textbf{depth of $x$ with respect to $\mathbf{N}$}, or simply, the \textbf{depth of $x$} when $\mathbf{N}$ is understood, is the number of points of $J\backslash\{x\}$ contained in the critical line for any minimal escape arc $J$ of $x$ with respect to $\mathbf{N}$. If $S$ is a subset of the complement of $\psi_\mathbf{N}(K)$, then we say that the \textbf{depth of $S$ with respect to $\mathbf{N}$} is the maximum of the depths of all points contained in $S$ if such maximum exists.\footnote{If for every $n \in \N$ there is a point $x$ in $S$ such that the depth of $x$ is greater than $n$, we say that the depth of $S$ is $\infty$. Such is the case for the complement of $\psi_\mathbf{N}(K)$ whenever $\mathbf{N}$ contains an increasing subsequence. However, we need not concern ourselves with infinite depth in the remainder of this paper.}
\end{definition}

\begin{remark}\label{alternate depths}
Adding to Remark \ref{alternate escapes}, we may wish to specify the depth a point $x$ has within such a region $R$. If $x$ is contained in $R$, we may say that the depth of $x$ in $R$ is the minimum number of times any minimal escape arc of $x$ from $R$ crosses the critical line in $R$. If $x$ is not contained in $R$, we may take the convention to be that the depth of $x$ in $R$ is 0.  However, this does not mean that the depth of $x$ with respect to $\mathbf{N}$ is 0. This is because $x$ may be in another pocket $R'$ of the complement of $\psi_\mathbf{N}(K)$ in which the depth of $x$ in $R'$ is not 0.
\end{remark}

The depth of a point in the complement of $\psi_\mathbf{N}(K)$ provides a means to measure how "trapped" or "confined" it is in the pockets of $\R^2\backslash \psi_\mathbf{N}(K)$. Note that a point $x$ may be contained in a pocket of the complement of $\psi_\mathbf{N}(K)$ but still have depth 0. This is because a minimal escape arc of $x$ would not have to cross the critical line $y = 1/2$ to be adjoined to a straight vertical ray. We should also note that a point $x$ in the complement of $\R^2$ not contained in schema or alternate schema pockets will also have a depth of 0 as there will always be a strait vertical ray having $x$ as an endpoint.



Again, we advise the reader to allow geometric intuition as a sketch of proof for the next two propositions, making use of Figures \ref{fig:subschema} and \ref{fig:escape} as needed.

\begin{proposition}\label{turning points to endpoints}
Let $\psi_\mathbf{N}$ be a schema embedding of $K$ and let $(x_1, x_2, x_3, \allowbreak \ldots)$ be a sequence of points so that for each $i \in \N$, $x_i \in D_{\mathbf{N},i}$ with the depth of $x_i$ in $D_{\mathbf{N},i}$ being at least 4.  If for each $i \in \N$, $J_i$ is a minimal escape arc from $x_i$ out of $D_{\mathbf{N},i}$, then the top turning points of $J_i$ converge to the top endpoint, $(0,1)$, of $\psi_\mathbf{N}(E)$ as $i \to \infty$ and the bottom turning points of $J_i$ converge to the bottom endpoint, $(0,0)$, of $\psi_\mathbf{N}(E)$ as $i \to \infty$. 
\end{proposition}

\begin{proposition}\label{depth of pockets}
Let $\psi_\mathbf{N}$ be the schema embedding of $K$. Then for each $i \in \N$, the depth of $D_{\mathbf{N},i}$ with respect to $\mathbf{N}$ is $n_i$, and the depth of $D_{\mathbf{N},i}'$ with respect to $\mathbf{N}$ is 2. Furthermore, if $L$ is a $V$-layer of $\Lambda$ layer of $K$, then the depth of $D_{\mathbf{N},L}$ with respect to $\mathbf{N}$ is at least 1. 
\end{proposition}

\begin{lemma}\label{V pockets to V pockets}
Let $\psi_\mathbf{N}$ and $\psi_\mathbf{M}$ be schema embeddings of $K$. Let $(V_1, V_2, V_3, \allowbreak \ldots)$ be a sequence of $V$-layers of $K$ and let $(x_1, x_2, x_3, \ldots)$ be a sequence of points in the plane converging to $\psi_\mathbf{N}(p) = (0,0)$ so that for each $i \in \N$, $x_i \in D_{\mathbf{N},V_i}$. Let $h$ be a homeomorphism of the plane onto itself mapping $\psi_\mathbf{N}(K)$ onto $\psi_\mathbf{M}(K)$ such that $h(\psi_\mathbf{N}(p)) = \psi_\mathbf{M}(p)$. Then there exists an $N \in \N$ such that for every $i \geq N$, $h(x_i)$ is contained in the $V$-pocket $D_{\mathbf{M},\Gamma_i}$, where $\Gamma_i := \psi^{-1}_\mathbf{M}(h(\psi_\mathbf{N}(V_i)))$.
\end{lemma}
\begin{proof}
Suppose there exists a subsequence $(x_{i_1}, x_{i_2}, x_{i_3}, \ldots )$ such that for every $j \in \N$, $h(x_{i_j}) \notin D_{\mathbf{M},\Gamma_{i_j}}$. By Corollary \ref{like layers} and Lemma \ref{homeos of K}, $h(\psi_\mathbf{N}(V_{i_j}))$ is a $V$-layer of $\psi_\mathbf{M}(K)$, and thus, $\Gamma_{i_j}$ a $V$-layer of $K$, for every $j  \in \N$.
We must consider two cases.

The first case we consider is that for each $j \in \N$, the depth of $h(x_{i_j})$ with respect to $\mathbf{M}$ is 0.  Since $x_{i_j} \to \psi_\mathbf{N}(p)$ as $j \to \infty$, it follows that $h(x_{i_j}) \to \psi_\mathbf{M}(p)$ as $j \to \infty$. Thus, for some $M \in \N$ and for every $j \geq M$, there exists a straight vertical arc $A_j$ whose top endpoint is $h(x_{i_j})$ so that $A_j \to A = \{0\} \times [-1,0]$ as $j \to \infty$. Since $A \cap \psi_\mathbf{M}(E) = \{\psi_\mathbf{M}(p)\}$, it follows that $h^{-1}(A) \cap \psi_\mathbf{N}(E) = \{\psi_\mathbf{N}(p)\}$. However, since $x_{i_j} \in D_{\mathbf{N}, V_{i_j}}$ for every $j \geq N$, it follows that $h^{-1}(A_j) \not\to h^{-1}(A)$ as $j \rightarrow \infty$, a contradiction to $h$ being a homeomorphism.

The second case considered is that for each $j \in \N$, the depth of $h(x_{i_j})$ with respect to $\mathbf{M}$ is positive. For each $j \in \N$, let $C_j$ be an endcut of $\psi_\mathbf{N}(K)$ contained in $\cl(D_{\mathbf{N},V_i})$ whose endpoints are $x_{i_j}$ and the vertex of $V_{i_j}$, and so that $C_j \to \{\psi_\mathbf{N}(p)\}$ as $j \to \infty$. Since $h(x_{i_j}) \notin D_{\mathbf{M}, \Gamma_{i_j}}$ for each $j \in \N$, it follows that $h(C_j) \not\to \{\psi_\mathbf{M}(p)\}$ as $j \rightarrow \infty$, also contradicting that $h$ is a homeomorphism. 
\end{proof}


We are now in position to provide a special case in which we prove that the standard embedding of $K$ is not equivalent to the schema embedding of $K$ according to the sequence of all 4's.  Recall by Remark \ref{remark on n} that this embedding can be denoted by $\psi_\mathbf{4}$. It follows as a consequence that the standard embedding of $K$ is inequivalent to any schema embedding of $K$.

\begin{lemma}\label{1 is not 4}
The standard embedding of the Knaster $V \Lambda$-continuum $K$ is inequivalent to the embedding $\psi_\mathbf{4}$. Moreover, the standard embedding of $K$ is inequivalent to $\psi_\mathbf{N}$ for any sequence $\mathbf{N}$ of positive even integers greater than or equal to 4.
\end{lemma}
\begin{proof}
For simplicity, we will denote the image of the standard embedding of $K$ as $X$, and the image of the embedding of $K$ under $\psi_\textbf{4}$ will be denoted as $Y$. Denote the left end layer of $X$ as $E_X$, the left end layer of $Y$ as $E_Y$, and denote the bottom and top endpoints of $E_X$ as $p_X$ and $q_X$ and the bottom and top endpoints of $E_Y$ as $p_Y$ and $q_Y$.\footnote{Even though $E_X = \{0\} \times [0,1] = E_Y$, $p_X = (0,0) = p_Y$, and $q_X = (0,1) = q_Y$, we still wish to make distinctions in reference to their corresponding embeddings.}  Recall from Remark \ref{remark on 1} that $X = \psi_\textbf{1}(K)$, where $\mathbf{1}$ here denotes the sequence of all 1's.  

Suppose $h$ is a homeomorphism of the plane onto itself so that $h(X) = Y$.  By Lemma \ref{homeos of K} and due to the symmetry of $X$, we may assume that $h(E_X) = E_Y$ with $h(p_X) = p_Y$ and $h(q_X) = q_Y$.  Let $(V_1, V_2, V_3, \ldots)$ be the sequence of $V$-layers of $Y$ whose vertices are $\psi_\mathbf{4}(s_i^{(4)})$ for each $i \in \N$.  Let $(y_1, y_2, y_3, \ldots)$ be a sequence of points such that for each $i \in \N$, $y_i$ is contained in the $V_i$-pocket of the complement of $Y$, with $y_i \to p_Y$ as $i \to \infty$ and so that $y_i \in D_{\mathbf{4},i}$, with the depth of $y_i$ being 4. 
Also as a consequence of Lemma \ref{homeos of K}, $h^{-1}(V_i)$ is a $V$-layer of $X$ for every $i \in \N$.

Again, by Remark \ref{remark on 1}, we can let, for each $i \in \N$, the horizontal crosscuts $C_{\mathbf{1},i}$ be connecting the endpoints of $h^{-1}(V_i)$ so that $D_{\mathbf{1},i}$ is the $i^\textnormal{th}$ schema pocket for $X$ having $h^{-1}(V_i)$ on its boundary.  That is, for each $i \in \N$, $D_{\mathbf{1},i}$ is the $V$-pocket, $D_{\mathbf{1},h^{-1}(V_i)}$, of $X$. For each $i \in \N$, let $x_i = h^{-1}(y_i)$.  Since $y_i \to p_Y$ as $i \to \infty$, it follows that $x_i \to p_X$ as $i \to \infty$.  By Lemma \ref{V pockets to V pockets}, we may assume for each $i \in \N$ that $x_i$ has a depth of 1 inside $D_{\mathbf{1},i}$. 

Let $J_i$ be the straight vertical arc so that $x_i$ is the bottom endpoint of $J_i$, and so that the top endpoint, $x_i'$, of $J_i$ has $1$ as its $y$-coordinate.  Note that in this case, $J_i \to E_X$ as $i \to \infty$. Furthermore, $J_i$ is a minimal escape arc from $x_i$ out of $D_{\mathbf{1},i}$ for each $i \in \N$. For each $i \in \N$, let $A_i$ be the straight arc of length 1 with $x_i'$ as its bottom endpoint. Then $A_i \to A$ where $A = \{0\} \times [1,2]$ as $i \to \infty$.

Since $h(x_i) = y_i \in D_{\mathbf{4},i}$ for each $i \in \N$, it follows that $h(J_i) \cap D_{\mathbf{4},i} \neq \emptyset$ for each $i \in \N$.  Since $h(J_i) \to E_Y$ with the endpoint $y_i$ of $h(J_i)$ having a depth of $4$ in $D_{\mathbf{4},i}$ for each $i \in \N$, and since by Lemma \ref{small bends}, each $h(J_i)$ has no points of depth less than 3 for all large enough $i$, it follows that there is $M \in \N$ such that for every $i \geq M$, the depth of $h(A_i)$ in $D_{\mathbf{4},i}$ is at least 3. Note that since $A \cap E_X = \{q_X\}$ and $\diam(A) = 1$, it follows there exists an $\eta > 0$ such that $\diam(h(A)) = \eta$ and $h(A) \cap E_Y = \{q_Y\}$. However, since $D_{\mathbf{4},i} \rightarrow E_Y$ as $i \rightarrow \infty$, and because the depth of $h(A_i)$ in $D_{\mathbf{4},i}$ is at least 3 for every $i \geq M$, it follows that $h(A_i) \not\to h(A)$. This contradicts that $h$ is a homeomorphism.

Since for each sequence $\mathbf{N}$ of positive even integers greater than or equal to 4, the depth the schema pocket $D_{\mathbf{N},i}$ is greater than or equal to 4 for every $i \in \N$, it follows that the standard planar embedding of $K$ is inequivalent to $\psi_{\mathbf{N}}$.
\end{proof}

\begin{corollary}\label{E to E}
If $\mathbf{N}$ and $\mathbf{M}$ are sequences of positive even integers greater than or equal to 4 and $h$ a homeomorphism of the plane onto itself mapping $\psi_\mathbf{N}(K)$ onto $\psi_\mathbf{M}(K)$, then $h(\psi_\mathbf{N}(E)) = \psi_\mathbf{M}(E)$.
\end{corollary}
\begin{proof}
Suppose instead that $h(\psi_\mathbf{N}(E)) = \psi_\mathbf{M}(E')$. Then there exists an $N \in \N$ for which 
$$h\Bigg(\psi_\mathbf{N}\bigg(\displaystyle \cl\Big( \bigcup_{i = N}^\infty K_i\Big)\bigg)\Bigg) \subset \psi_\mathbf{M}(K_1^{(1)}).$$
Let $\mathbf{N}'$ denote the subsequence of $\mathbf{N}$ having all but the first $N-1$ terms of $\mathbf{N}$.  Then $\mathbf{N}'$ and $\mathbf{N}$ are equivalent sequences, and thus, $\psi_\mathbf{N}$ and $\psi_\mathbf{N'}$ are equivalent planar embeddings of $K$ by Proposition \ref{equiv seq equiv emb}.  Since $\psi_\mathbf{N'}$ is equivalent to $\psi_\mathbf{N} \upharpoonright \cl\big( \bigcup_{i = N}^\infty K_i\big)$, so is $\psi_\mathbf{N}$.  However, $\psi_\mathbf{M} \upharpoonright K_1^{(1)}$ is equivalent to the standard embedding of $K$, with $\psi_\mathbf{N}\Big(\displaystyle \cl\big( \bigcup_{i = N}^\infty K_i\big)\Big)$ being mapped by $h$ into $\psi_\mathbf{M}(K_1^{(1)})$, contradicting Lemma \ref{1 is not 4}.
\end{proof}

\begin{lemma}\label{p to p}
If $\mathbf{N}$ and $\mathbf{M}$ are sequences of positive even integers greater than or equal to 4 and $h$ a homeomorphism of the plane onto itself mapping $\psi_\mathbf{N}(K)$ onto $\psi_\mathbf{M}(K)$, then $h(\psi_\mathbf{N}(p)) = \psi_\mathbf{M}(p)$.
\end{lemma}
\begin{proof}
For simplicity, let us denote $\psi_\mathbf{N}(p)$, $\psi_\mathbf{N}(q)$, $\psi_\mathbf{M}(p)$, and $\psi_\mathbf{M}(q)$ by $p_\mathbf{N}$, $q_\mathbf{N}$, $p_\mathbf{M}$, and $q_\mathbf{M}$, respectively. By Corollary \ref{E to E}, $h(p_\mathbf{N}) \in \{p_\mathbf{M},q_\mathbf{M}\}$.  Suppose $h(p_\mathbf{N}) = q_\mathbf{M}$.  Let $(x_1, x_2, x_3, \ldots)$ be a sequence of points converging to $p_\mathbf{N}$ with $x_i$ having a depth of 4 in $D_{\mathbf{N},i}$, and for each $i \in \N$, let $H_i$ be a sequence of minimal escape arcs from $x_i$ out of $D_{\mathbf{N},i}$.  We shall also assume that $x_i$ has the least $y$-coordinate among all other points of depth 4 along $H_i$ for each $i \in \N$.  Also for each $i \in \N$, let $x_i^{(l)}$ for $l = 0, 1, 2, 3, 4$ be the points along $H_i$ so that
    \begin{itemize}
        \item[(1)] $x_i^{(0)} = x_i$, and $x_i^{(4)}$ is the other endpoint of $H_i$ intersecting the crosscut $C_{\mathbf{N},i}$,
        \item[(2)] $x_i^{(l)}$ has depth $4-l$ in $D_{\mathbf{N},i}$ for each $l = 0, 1, 2, 3, 4$, and
        \item[(3)] $x_i^{(l)}$ is a top turning point of $H_i$ for $l = 1, 3$ and $x_i^{(2)}$ is a bottom turning point.
    \end{itemize}
    
Note that since $h(p_\mathbf{N}) = q_\mathbf{M}$, and thus, $h(q_\mathbf{N}) = p_\mathbf{M}$, condition $(3)$ gives us $h(x_i^{(2)}) \to q_\mathbf{M}$ and $h(x_i^{(l)}) \to p_\mathbf{M}$ for $l = 1, 3$ as $i \to \infty$ by Proposition \ref{turning points to endpoints}. Also, by default, we have $x_i^{(l)} \to p_\mathbf{N}$ as $i \to \infty$ for $l = 0, 4$, whence $h(x_i^{(l)}) \to q_\mathbf{M}$ as $i \to \infty$ for $l = 0,2,4$. If for each $i \in \N$ we let $H_i^{(l)}$ be the subarc of $H_i$ having as its endpoints $x_i^{(l-1)}$ and $x_i^{(l)}$ for each $l = 1, 2, 3, 4$, then for each such $l$, $(H_1^{(l)}, H_2^{(l)}, H_3^{(l)}, \ldots)$ forms a sequence of arcs converging to $\psi_\mathbf{N}(E)$ so that for each $i \in \N$, $\pi \upharpoonright H_i^{(l)}$ is a homeomorphism mapping $H_i^{(l)}$ onto $\pi(H_i^{(l)})$. It follows that, if $U$ and $W$ are the open subsets of the $xy$-plane sitting respectively above and below the critical line $y = 1/2$, then there exists an $N \in \N$ such that for every $i \geq N$, 
$$\mathcal{H}_i = \{h(x_i^{(0)}), h(x_i^{(1)}), h(x_i^{(2)}), h(x_i^{(3)}), h(x_i^{(4)})\} \subset U \cup W,$$
with members of the above set alternating between belonging to $U$ and belonging to $W$. That is, for each $i \geq N$, $h(x_i^{(l)}) \in U$ for $l = 0, 2, 4$ and $h(x_i^{(l)}) \in W$ for $l = 1, 3$.

\textit{Claim.} Only finitely many $H_i$ are mapped by $h$ into schema pockets of the complement of $\psi_\mathbf{N}(K)$ so that for each such $i$, there is a pair of distinct integers $l_i$ and $l_i'$ such that the depth of $h(x_i^{(l_i)})$ is the same as the depth of $h(x_i^{(l_i')})$ 


\textit{Proof of Claim.} Suppose there is a subsequence $i_1, i_2, i_3, \ldots$ of positive integers so that for each $j \in \N$, there exists two distinct integers $l_j$ and $l_j'$ between 0 and 4 such that the depth of $h(x_{i_j}^{(l_j)})$ is the same as that of $h(x_{i_j}^{(l_j')})$. Since members of the set $\mathcal{H}_{i_j}$ alternate between being in $U$ and being in $W$ for every $i_j \geq N$, it follows that $l_j$ and $l_j'$ can be chosen so that $|l_j -l_j'| = 2$ for every $i_j \geq N$. As a consequence of Lemma \ref{small bends}, we may consider two cases.

The first case is that for each $j$ such that $i_j \geq N$, there exists an arc $T_j$ contained in  $(\R^2\backslash\psi_\mathbf{M}(K)) \cap U$ or contained in $(\R^2\backslash\psi_\mathbf{M}(K)) \cap W$ whose endpoints are $h(x_{i_j}^{(l_j)})$ and $h(x_{i_j}^{(l_j')})$, and with $\diam(T_j) \to 0$ as $j \to \infty$. However, since for each $j \in \N$, $x_{i_j}^{(l_j)}$ and $x_{i_j}^{(l_j')}$ have depths differing by 2, it follows that $\diam(h^{-1}(T_j)) \not\to 0$ as $j \rightarrow \infty$---a contradiction to $h$ being a homeomorphism.

The second to be considered is that for some $M \geq N$ and every $j \geq M$, there is a straight vertical arc $A_j \subset (\R^2\backslash\psi_\mathbf{M}(K)) \cap W$ whose top endpoint is $h(x_{i_j}^{(l_j)})$, with $l_j \in \{1,3\}$, such that $A_j \to A = \{0\} \times [-1,0]$ as $j \to \infty$. Note that $A \cap E_Y = \{p_Y\}$ implies that $h^{-1}(A) \cap E_X = \{p_X\}$.  However, since for each $j \geq M$, the depth of $x_{i_j}^{(l_j)}$ is at least 1, it follows that $h^{-1}(A_j) \not\to h^{-1}(A)$.  This also contradicts that $h$ is a homeomorphism.

By our previous claim, it follows that for some $N' \geq N$ and every $i \geq N'$, no two different $h(x_i^{(l)})$ have the same depth within a schema pocket of the complement of $\psi_\mathbf{M}(K)$. By the previous claim combined with Proposition \ref{small bends}, it follows that for each $i \geq N'$, $h(H_i)$ is mapped into a schema pocket of the complement of $\psi_\mathbf{M}(K)$ such that the depths of each $h(x_i^{(l)})$ alternate in value, with the depths of $h(x_i^{(l)})$ and $h(x_i^{(l')})$ differing by 1 if and only if $|l-l'| = 1$.  Furthermore, the depth of each such $h(x_i^{(l)})$ is no less than $5 - l$.  Thus, in particular, for each $i \geq N'$, the depth of $h(x_i^{(4)})$ will be at least 1.

For each $i \in \N$, let $A_i$ be the straight vertical arc whose top endpoint is $x_i^{(4)}$, with $A_i \to A$, where again, $A = \{0\} \times [-1,0]$, as $i \to \infty$. Since $A \cap E_X = \{p_X\}$, it follows that $h(A) \cap E_Y = \{p_Y\}$.  However, since the depth of $h(x_i^{(4)})$ is at least 1 for each $i \geq N'$, it follows that $h(A_i) \not\to h(A)$ as $i \to \infty$---a contradiction to $h$ being a homeomorphism. Therefore, $h(p_\mathbf{N}) = p_\mathbf{M}$.
\end{proof}

\subsection{Schema Embeddings of $K$ According to Inequivalent Sequences are Inequivalent}

\

We now state and prove the lemmas needed to show that the collection of all schema embeddings of $K$ according to sequences of even positive integers greater than or equal to 4 has cardinality $\mathfrak{c}$. In what follows, let $\mathbf{N} = (n_1, n_2, n_3, \ldots)$ and $\mathbf{M} = (m_1, m_2, m_3, \ldots)$ be sequences of even positive integers greater than or equal to 4. We will let $X = \psi_\mathbf{N}(K)$ and $Y = \psi_\mathbf{M}(K)$ for simplicity.  Though both are equal to $\{0\} \times [0,1]$, we will let the left end layers of $X$ and $Y$ be denoted by $E_X$ and $E_Y$, respectively. Denote the bottom and top endpoints of $E_X$ as $p_X$ and $q_X$, respectively, and the bottom and top points of $E_Y$ as $p_Y$ and $q_Y$, respectively. We remind the reader that $\pi$ denotes the natural projection of $\R^2$ onto the $y$-axis so that $\pi(x,y) = (0,y)$ for every $(x,y) \in \R^2$.

Let $x_i$ be a point in $D_{\mathbf{N},i}$ so that the depth of $x_i$ is $n_i$ for each $i \in \N$, and with $x_i \to p_X$ as $i \to \infty$. We will also assume that for each $i \in \N$, $x_i$ has the smallest $y$-coordinate for any minimal escape arc from $x_i$ out of $D_{\mathbf{N},i}$. For each $i \in \N$, let $J_i$ be a minimal escape arc from $x_i$ out of $D_{\mathbf{N},i}$, and let $x_{i}^{(1)}, \ldots, x_{i}^{(n_i-1)} \in D_{\mathbf{N},i}$ and $x_{i}^{(n_i)} \in C_{\mathbf{N},i}$ be such that the depth of $x_i^{(l)}$ in $D_{\mathbf{N}, i}$ is $n_i - l$ for each $l = 1, \ldots, n_i$. Furthermore, let $x_{i}^{(1)}, \ldots, x_{i}^{(n_i-1)}$ be the turning points of $J_i$.  We will designate $x_i^{(0)} := x_i$ for each $i \in \N$. Let $J_{i}^{(1)}, \ldots, J_{i}^{(n_i)}$ be such that $J_{i}^{(l)}$ is the subarc of $J_i$ whose endpoints are $x_{i}^{(l-1)}$ and $x_{i}^{(l)}$ for every $l = 1, \ldots, n_i$.

Similarly, let $y_i$ be a point in $D_{\mathbf{M},i}$ so that the depth of $y_i$ is $m_i$ for each $i \in \N$, and with $y_i \to p_Y$ as $i \to \infty$. We will also assume that for each $i \in \N$, $y_i$ has the smallest $y$-coordinate for any minimal escape arc from $y_i$ out of $D_{\mathbf{M},i}$. For each $i \in \N$, let $I_i$ be a minimal escape arc from $y_i$ out of $D_{\mathbf{M},i}$, and let $y_{i}^{(1)}, \ldots, y_{i}^{(m_i-1)} \in D_{\mathbf{M},i}$ and $y_{i}^{(m_i)} \in C_{\mathbf{M},i}$ be such that the depth of $y_i^{(l)}$ in $D_{\mathbf{M}, i}$ is $m_i - l$ for each $l = 1, \ldots, m_i$. Furthermore, let $y_{i}^{(1)}, \ldots, y_{i}^{(m_i-1)}$ be the turning points of $I_i$.  We will designate $y_i^{(0)} := y_i$ for each $i \in \N$. Let $I_{i}^{(1)}, \ldots, I_{i}^{(m_i)}$ be such that $I_{i}^{(l)}$ is the subarc of $I_i$ whose endpoints are $y_{i}^{(l-1)}$ and $y_{i}^{(l)}$ for every $l = 1, \ldots, m_i$.

We shall also denote by $U$ and $W$ the open subsets of $\R^2$ sitting above and below, respectively, the critical line $y = 1/2$. Lastly, we will assume $h$ is a homeomorphism onto itself such that $h(X) = Y$.  By Corollary \ref{E to E} and Lemma \ref{p to p}, $h(E_X) = E_Y$ with $h(p_X) = p_Y$.

\begin{lemma}\label{smaller depth}
All but finitely many $h(x_i)$ are contained in schema pockets of the complement of $Y$ in which the depth of $h(x_i)$ is at least $n_i$.
\end{lemma}
\begin{proof}
Suppose there is a subsequence $(x_{i_1}, x_{i_2}, x_{i_3}, \ldots)$ having the property that for each $j \in \N$, $h(x_{i_j})$ has depth within any schema pocket of the complement of $Y$ being less than $n_i$. Since $x_i \to p_X$ as $i \to \infty$, it follows that $h(x_{i_j}) \to p_Y$ as $j \to \infty$, and thus, there exists an $N \in \N$ such that for every $j \geq N$, the depth of $h(x_{i_j})$ in a schema pocket of the complement of $Y$ is even. $N$ can be taken large enough so that $h(x_{i_j}^{(l)}) \in U$ when $l$ is odd and $h(x_{i_j}^{(l)}) \in W$ when $l$ is even. We must now consider two cases.

The first case we consider is that there is an $M \geq N$ such that for each $j \geq M$, there exists an arc $T_j$ contained in $(\R^2\backslash\psi_\mathbf{M}(K)) \cap U$ or $(\R^2\backslash\psi_\mathbf{M}(K)) \cap W$ having as its endpoints $h(x_{i_j}^{(l_j)})$ and $h(x_{i_j}^{(l_j')})$, with $\diam(T_j) \allowbreak \to 0$ as $j \to \infty$, where $l_j, l_j' \in \{0, 1, 2, \ldots, n_{i_j}\}$ and $|l_j - l_j'| = 2$. However, since the difference in the depths of $x_{i_j}^{(l_j)}$ and $x_{i_j}^{(l_j')}$ is 2, this implies that $\diam(h^{-1}(T_j)) \not\to 0$ as $j \to \infty$. This contradicts that $h$ is a homeomorphism. 

The second case we consider is that for some $M \geq N$ and every $j \geq M$, there is a straight vertical arc $A_j$ whose top endpoint is $h(x_{i_j}^{(l_j)})$, with $A_j \to A = \{0\} \times [-1,0]$ as $j \to \infty$ for an even $l_j \in \{0, 2, \ldots, n_{i_j}-2\}$. Recall from previous proofs that since $A \cap E_Y = \{p_Y\}$, it follows that $h^{-1}(A) \cap E_X = \{p_X\}$. However, since for each $j \geq M$, the depth of $x_{i_j}^{(l_j)}$ is at least 2, it follows that $h^{-1}(A_j)\not\to h^{-1}(A)$ as $j \to \infty$. This is also a contradiction to $h$ being a homeomorphism.
\end{proof}

By using similar arguments in the proof of Lemma \ref{smaller depth}, one may obtain the following corollary.

\begin{corollary}\label{depth of turns}
All but finitely many $h(x_i^{(l)})$ are contained in schema pockets of the complement of $Y$ in which the depth of $h(x_i^{(l)})$ is at least $n_i - l$, where $l \in \{0, 1, \ldots, n_i\}$. 
\end{corollary}

\begin{lemma}\label{larger depth}
All but finitely many $h(x_i)$ are not contained in schema pockets of the complement of $Y$ whose depth is more than $n_i$.
\end{lemma}
\begin{proof}
Suppose there is a subsequence $(x_{i_1}, x_{i_2}, x_{i_3}, \ldots)$ such that $h(x_{i_j})$ is in a schema pocket, $D_{\mathbf{M},k(j)}$, of the complement of $Y$ whose depth is greater than $n_{i_j}$.  By Lemma \ref{smaller depth}, this leaves two possibilities:
    \begin{itemize}
        \item[(1.)] The depth of $h(x_{i_j})$ in $D_{\mathbf{M},k(j)}$ is $n_{i_j}$.
        \item[(2.)] The depth of $h(x_{i_j})$ in $D_{\mathbf{M},k(j)}$ is greater than $n_{i_j}$.
    \end{itemize}

Assume case (1.) occurs for all but finitely many $j \in \N$ and, without loss of generality, for every $j \in \N$. As in the proof of Lemma \ref{smaller depth}, there exists an $N \in \N$ such that for every $j \geq N$, $\{h(x_{i_j}), h(x_{i_j}^{(1)}), \ldots, h(x_{i_j}^{(n_{i_j})})\} \subset U \cup W$, the members of this set alternating between $U$ and $W$. By Lemma \ref{small bends}, we may assume $N$ is large enough so that for each $j \geq N$, the difference between the depths of $h(x_{i_j}^{(l)})$ and $h(x_{i_j}^{(l')})$ in $D_{\mathbf{M},k(j)}$ is $1$ if and only if $|l - l'| = 1$.  Furthermore, by Corollary \ref{depth of turns}, the depth of every $h(J_{i_j}^{(l)})$ is that of $h(x_{i_j}^{(l-1)})$, which is $n_{i_j} - l + 1$, for each $j \geq N$ and each $l \in \{ 1, \ldots, n_{i_j}\}$.

For each $j \geq N$, let $z_j$ be a point in $D_{\mathbf{M},k(j)}$ whose depth is $n_{i_j}+1$ and so that $z_j \to q_Y$ as $j \to \infty$. Note that $h^{-1}(z_j) \to q_X$ as $j \to \infty$.  Also for each $j \geq N$, let $Z_j$ be an arc in $D_{\mathbf{M},k(j)}$ whose endpoints are $h(x_{i_j})$ and $z_j$ and so that the depth of $Z_j$ in $D_{\mathbf{M},k(j)}$ is also $n_{i_j}+1$. 

It follows, as a consequence of Lemma \ref{small bends} on the collection of all $Z_j$, that there exists a positive integer $M \geq N$ such that for every $j \geq M$, there exists an arc $T_j$ in $D_{\mathbf{N},i_j}$ whose endpoints are $h^{-1}(z_j)$ and $x_{i_j}^{(1)}$, with $\diam(T_j) \rightarrow 0$ as $j \to \infty$.  However, since the depth of $z_j$ in $D_{\mathbf{M},k(j)}$ differs by $2$ from the depth of $h(x_{i_j}^{(1)})$ in $D_{\mathbf{M},k(j)}$ for every $j \geq M$, it follows that $\diam(h(T_j)) \not\to 0$ as $j \to \infty$, a contradiction to $h$ being a homeomorphism.  

Assume now that case (2.) occurs for all but finitely many $j \in \N$ and, without loss of generality, for every $j \in \N$. Let $N$ be similarly defined as in the proof negating case (1.) above.  Since the depth of $h(x_{i_j})$ in $D_{\mathbf{M},k(j)}$ is greater than $n_{i_j}$ for every $j \in \N$, this implies that for every $j \geq N$, the depth of $h(x_{i_j}^{(n_{i_j})})$ in $D_{\mathbf{M},k(j)}$ is greater than or equal to $1$.  

For each $j \geq N$, let $A_j$ be the straight vertical arc having $x_{i_j}^{(n_{i_j})}$ as its top endpoint with $A_j \to A = \{0\} \times [-1,0]$ as $j \to \infty$. Again, since $A \cap E_X = \{p_X\}$, it follows that $h(A) \cap E_Y = \{p_Y\}$.  However, for each $j \geq N$, since the depth of $h(x_{i_j}^{(n_{i_j})})$ in $D_{\mathbf{M},k(j)}$ is greater than or equal to $1$, this implies $h(A_j) \not \to h(A)$ as $j \to \infty$. This also contradicts that $h$ is a homeomorphism.   
\end{proof}

\begin{lemma}\label{different pockets}
All but finitely many pairs $h(x_i)$ and  $h(x_j)$, where $i \neq j$, are contained in different schema pockets of the complement of $Y$.
\end{lemma}
\begin{proof}
Suppose there exists a subsequences $(x_{i_1}, x_{i_2}, x_{i_3}, \ldots)$ and a subsequence $(x_{j_1}, x_{j_2}, x_{j_3}, \ldots)$ such that $i_m \neq j_m$ for every $m \in \N$ and so that $h(x_{i_m})$ and $h(x_{j_m})$ are contained in the same schema pocket of the complement of $Y$.  Again, as in the previous proofs, there exists an $M(i) \in \N$ such that for every $m \geq M(i)$, $\{h(x_{i_m}), h(x_{i_m}^{(1)}), \ldots, h(x_{i_m}^{(n_{i_m})})\} \subset U \cup W$, the members of this set alternating between $U$ and $W$.  Similarly, there exists an $M(j) \in \N$ such that for every $m \geq M(j)$, $\{h(x_{j_m}), h(x_{j_m}^{(1)}), \ldots, h(x_{j_m}^{(n_{j_m})})\} \subset U \cup W$, the members of this set alternating between $U$ and $W$. Furthermore, for each $m \geq M(i)$, the difference between the depths of $h(x_{i_m}^{(l)})$ and $h(x_{i_m}^{(l')})$ is $1$ if and only if $|l - l'| = 1$, with the depth of every $h(J_{i_m}^{(l)})$ being that of $h(x_{i_m}^{(l-1)})$, and for each $m \geq M(j)$, the difference between the depths of $h(x_{j_m}^{(l)})$ and $h(x_{j_m}^{(l')})$ is $1$ if and only if $|l - l'| = 1$, with the depth of every $h(J_{j_m}^{(l)})$ is being that of $h(x_{j_m}^{(l-1)})$.  

Let $M = \max\{M(i), M(j)\}$.  By Lemma \ref{smaller depth} and Lemma \ref{larger depth}, the depth of $h(x_{i_m})$ and $h(x_{j_m})$ are the same as the depths of $x_{i_m}$ and $x_{j_m}$ in $D_{\mathbf{N},i_m}$ and $D_{\mathbf{N},j_m}$, respectively, for every $m \geq M$.  That is, $n_{i_m} = n_{j_m}$ for every $m \geq M$.  In particular, for each $m \geq M$ and each $l = 1, \ldots, n_{i_m}$, the depth of $h(x_{i_m}^{(l)})$ is $n_{i_m} - l$, and the depth of $h(J_{i_m}^{(l)})$ is $n_{i_m} - l + 1$.  These are also the depths of the corresponding $h(x_{j_m}^{(l)})$ and $h(J_{j_m}^{(l)})$.   

For every $m \geq M$ and for every $l \in \{1, \ldots, n_{i_m}\}$, let $T_m^{(l)}$ be the shortest arc contained in the complement of $Y$ so that the endpoints of $T_m^{(l)}$ are $h(x_{i_m}^{(l)})$ and $h(x_{j_m}^{(l)})$. Then for every such $m$ and every such $l$, $\diam(T_m^{(l)}) \rightarrow 0$ as $m \rightarrow \infty$. However, because $i_m \neq j_m$ for every $m \geq M$, it follows that for every $m \geq M$ and every $l \in \{1, \ldots, n_{i_m}\}$, $\diam(h^{-1}(T_m^{(l)})) \not\to 0$ as $m \to \infty$. This contradicts that $h$ is a homeomorphism.
\end{proof}

\begin{lemma}\label{equivalent sequences}
$\mathbf{N}$ and $\mathbf{M}$ are equivalent sequences of positive integers.
\end{lemma}
\begin{proof}
By Lemmas \ref{smaller depth} and \ref{larger depth}, there exists a positive integer $N$ such that for every $i \geq N$, $h(x_i)$ is contained in a schema pocket $D_{\mathbf{M},j_i}$ of the complement of $Y$ with depth $n_i$ in which the depth of $h(x_i^{(l)}) = n_i - l$ for every $l \in \{0, 1, \ldots, n_i\}$ and so that the depth of $h(J_i^{(l)})$ is that of $h(x_i^{(l-1)})$.  Furthermore, such a sufficiently large $N$ can also satisfy $j_i \neq j_k$ whenever $i \neq k$ by Lemma \ref{different pockets}, with $j_i < j_{i+1}$ for every $i \geq N$ by Proposition \ref{preserve order}. Similarly, there exists a positive integer $M$ such that for every $i \geq M$, $h^{-1}(y_i)$ is contained in a schema pocket $D_{\mathbf{N},s_i}$ of the complement of $X$ with depth $m_i$ in which the depth of $h^{-1}(y_i^{(l)}) = m_i - l$ for every $l = 0, 1, \ldots, m_i$ and so that the depth of $h^{-1}(I_i^{(l)})$ is that of $h^{-1}(y_i^{(l-1)})$.  Furthermore, such a sufficiently large $M$ can also satisfy $s_i \neq s_k$ whenever $i \neq k$ by Lemma \ref{different pockets}, with $s_i < s_{i+1}$ for every $i \geq M$ by Proposition \ref{preserve order}.

Let us assume that $N = \max\{N,M\}$.  By the lemmas listed in the previous paragraph, it follows that for every $i \geq N$, $j_{i+1} - j_i = 1$.  Indeed, suppose that there is a sequence $i_1, i_2, i_3, \ldots$ such that $j_{i_k+1} - j_{i_k} > 1$ for every $k \in \N$.  Then there exists a sequence of integers $(u_1, u_2, u_3, \ldots)$ such that $j_{i_k} < u_k < j_{i_{k+1}}$ for each $k \in \N$.  By Proposition \ref{preserve order}, it follows that for every $k \in \N$, $h^{-1}(y_{u_k})$ must be mapped to a pocket of the complement of $X$ between $D_{\mathbf{N},i_k}$ and $D_{\mathbf{N},i_k+1}$.  However, as a consequence of Lemma \ref{different pockets} as well as Lemma \ref{V pockets to V pockets}, the only option would be for $h^{-1}(y_{u_k})$ to be mapped to a $V$-pocket of $X$ not contained in any schema disk of the complement of $X$.  Such a pocket would have a depth of at most 2, and since the depth of $y_{u_k}$ is at least 4 for each $k$, this would contradict Lemma \ref{smaller depth} and thus proves our claim.

It follows that for every $i \geq N$, pockets $D_{\mathbf{N},i}$ are mapped so that their depth in $D_{\mathbf{M},j_i}$ is $n_i$, with the depth of $D_{\mathbf{M},j_i}$ being $n_i$ as well.  Thus, for each $i \geq N$ and each $k \geq j_N$, we have that $n_i = m_k$.  Therefore, $\mathbf{N}$ and $\mathbf{M}$ are equivalent sequences of positive even integers greater than or equal to 4. 
\end{proof}

\begin{theorem}\label{inequivalent embeddings}
If $\mathbf{N}$ and $\mathbf{M}$ are inequivalent sequences of positive even integers greater than or equal to 4, then $\psi_\mathbf{N}$ and $\psi_\mathbf{M}$ are inequivalent planar embeddings of $K$.
\end{theorem}
\begin{proof}
This follows directly as a consequence of Lemma \ref{equivalent sequences}, where we preemptively assumed that $\psi_\mathbf{N}$ and $\psi_\mathbf{M}$ were equivalent planar embeddings of $K$ by assuming there was a homeomorphism $h$ of the plane onto itself mapping $\psi_\mathbf{N}(K)$ onto $\psi_\mathbf{M}(K)$, leading us to conclude that $\mathbf{N}$ and $\mathbf{M}$ are equivalent sequences.  Therefore, by contraposition, we conclude that if $\mathbf{N}$ and $\mathbf{M}$ are inequivalent sequences, $\psi_\mathbf{N}$ and $\psi_\mathbf{M}$ must be inequivalent planar embeddings of $K$.
\end{proof}

\begin{corollary}\label{uncountable embeddings}
There are $\mathfrak{c}$-many mutually inequivalent planar embeddings of $K$, each whose image has the same accessible points as in the standard embedding of $K$.
\end{corollary}
\begin{proof}
Let $\mathcal{I}$ be as in the proof of Lemma \ref{uncountable sequences}, and let $\mathcal{G} \subset \mathcal{I}$ be the set of all such equivalence classes of sequences of positive integers greater than or equal to 4.  By a proof similar to that of Lemma \ref{uncountable sequences}, the cardinality of $\mathcal{G}$ is $\mathfrak{c}$. Let $\mathfrak{G}$ be a set formed upon choosing one sequence from each equivalence class in $\mathcal{G}$.  Since by Theorem \ref{inequivalent embeddings} the embeddings $\psi_\mathbf{N}$ and $\psi_\mathbf{M}$ are inequivalent planar embeddings of $K$ whenever $\mathbf{N}$ and $\mathbf{M}$ are inequivalent sequences in $\mathfrak{G}$, it follows that $\Psi = \{\psi_\mathbf{N} \mid \mathbf{N} \in \mathfrak{G}\}$ is a collection of planar embeddings of $K$ whose cardinality is $\mathfrak{c}$.  Furthermore, the image of each member of $\Psi$ has the same set of accessible points as the standard embedding of $K$ by Proposition \ref{layers accessible}.
\end{proof}

\section{Closing Comments and Open Questions}

The question of which HDCC admit uncountably many mutually inequivalent planar embeddings is still an open problem. The previously described techniques for constructing $\mathfrak{c}$-many mutually inequivalent planar embeddings for $K$ may provide insight into how this more general problem can be solved. However, since any given HDCC, $X$, may possess a highly complex underlying structure on its generalized layers, producing planar embeddings of $X$ using similar such techniques can prove difficult in controlling the details of their constructions.

\

\noindent
\textbf{Question 1.} Can the techniques of producing $\mathfrak{c}$-many mutually inequivalent planar embeddings of $K$ as in Section 2 and Section 3 be generalized to all non-arc HDCC? If not for all non-arc HDCC, can they be generalized to those whose layer level is finite? 

\

We may also be able to provide a partial answer to the more general question above by exploring HDCC which yield a decomposition similar to that of $K = K_1 \cup K_2 \cup \cdots \cup E$ as in Section 3. In particular, suppose $X$ is an HDCC with left end layer $E_X$. Furthermore, suppose $X$ can be decomposed as $X = X_1 \cup X_2 \cup \cdots \cup E_X$, where $X_i \cap E_X = \emptyset$ for each $i \in \N$, and so that $X_i \cap X_j$ is a subcontinuum $C_i$ of $X$ if and only if $|i - j| \leq 1$, with $\diam(C_i) \to 0$ as $i \to \infty$. We thus pose the following question.

\

\noindent
\textbf{Question 2.} Can any non-arc HDCC possessing a decomposition as described in the previous paragraph be embedded in the plane in $\mathfrak{c}$-many mutually inequivalent ways? If so, can such embeddings be constructed according to sequences of positive even integers greater than or equal to 4 as was done for $K$ in section 3?

\

\begin{figure}
\centering
\includegraphics[width=8cm]{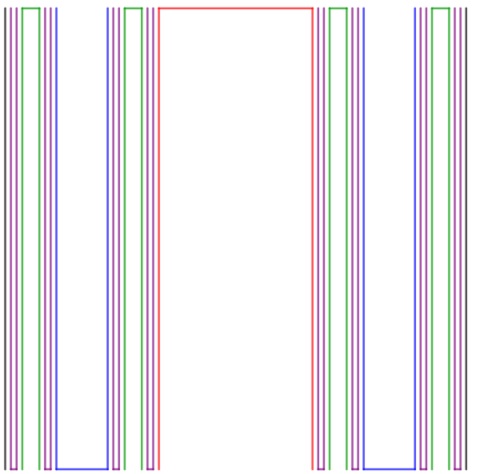}
\caption{The Cantor organ, $C$, represented by the first four iterations of its construction including its end layers. The layer level of $C$ is the same as that of $K$. In fact, $C$ can be constructed from $K$ by "blowing up" each individual vertex of $V$ and $\Lambda$ layers of $K$ into a horizontal arc. Likewise, $K$ can be constructed from $C$ by collapsing each individual horizontal arc in $C$ to a point.}
\label{fig:organ}
\end{figure}

There are many simple examples of HDCC which match the decomposition criteria mentioned in the paragraph before the previous question. The simplest such non-arc candidates are those possessing a subcontinuum which is the compactification of a ray with an arc. However, as we stated in section 1, and as was mentioned in \cite{Anusic2}, it is likely that such HDCC can be embedded in the plane in uncountably many mutually inequivalent ways. Another example of a non-arc HDCC other than $K$ matching the aforementioned decomposition criteria above is the continuum $C$ depicted in Figure \ref{fig:organ} known as the \textit{Cantor organ}. $C$ was given as an example of an irreducible space by Kuratowski in Chapter V, $\S 48$ of \cite{Kuratowski}. Though the Cantor set is invoked in its construction, Kuratowski did not refer to $C$ as the Cantor organ. However, one can find that Janusz J. Charatonik, Pawel Krupski, and Pavel Pyrih have given $C$ this name in \cite{Charatonik}. Regardless of the nomenclature, one may naturally propose a way to construct $\mathfrak{c}$-many mutually inequivalent planar embeddings of $C$ similar to the schema embeddings of $K$ in Section 3.

\

Recall the statements of Corollary \ref{like layers} and Lemma \ref{homeos of K}. To summarize, layers of $K$ of a given type are mapped to layers of the same type under a homeomorphism of $K$ onto itself, but it is not necessary that any layer is mapped onto itself.

\

\noindent
\textbf{Question 3.} Is there a planar embedding $\varphi$ of $K$ such that, for every layer $L$ of $K$ and every homeomorphism $h$ of the plane onto itself mapping $\varphi(K)$ onto itself, $h(\varphi(L)) = \varphi(L)$? If so, what other HDCC have this property?

\

It is possible that the question above can be answered for $K$ by inserting schema embeddings of subcontinuum copies of $K$ "all over $K$." However, the details of providing such an embedding of $K$ are yet to be developed. It may also be possible that if such an embedding can be constructed, there exist $\mathfrak{c}$-many mutually inequivalent such embeddings.

\end{document}